\documentclass[12pt,reqno]{amsart}

\usepackage{amsmath,amssymb,amsthm} 
\usepackage{graphicx} 
\usepackage{geometry} 
\usepackage{amsfonts,amsthm,amsmath} 
\usepackage{amssymb,mathrsfs,mathdots,latexsym} 

\usepackage[square,numbers,sort&compress]{natbib} 
\usepackage{array} 
\usepackage[colorlinks]{hyperref} 
\usepackage{booktabs} 
\usepackage{enumitem} 
\usepackage{subcaption} 
\usepackage{xcolor} 
\usepackage{microtype} 
\usepackage{algorithm} 
\usepackage{algorithmic} 
\usepackage{multicol} 
\usepackage{caption} 

\counterwithout{algorithm}{section}

\theoremstyle{plain}
\newtheorem{theorem}{Theorem}[section]
\newtheorem*{theorem*}{Theorem}

\newtheorem{lemma}[theorem]{Lemma}
\newtheorem{proposition}[theorem]{Proposition}
\theoremstyle{definition}
\newtheorem{example}[theorem]{Example}
\numberwithin{equation}{section}

\theoremstyle{remark}

\newtheorem{remark}[theorem]{Remark}

\usepackage{marvosym}

 \def\C{\mathbb C}
   
\def \N{\mathbb N}

\def\wtilde{\widetilde}

\newcommand{\coloneqq}{\mathrel{\mathop:}=}

\newcommand{\R}{\mathbb{R}}
\newcommand{\In}{\operatorname{In}}

\begin{document}

\title{A generalized Hurwitz stability criterion via rectangular block Hankel matrices for nonmonic matrix polynomials}

\author{Zixiang Ni}  
\address{
Department of Mathematics, Beijing Normal University at Zhuhai,
    Zhuhai 519087, PR China
}
\email{zxni@mail.bnu.edu.cn}

\author{Yongjian Hu}  
\address{
Department of Mathematics, Beijing Normal University at Zhuhai,
    Zhuhai 519087, PR China; School of Mathematical Sciences, Beijing Normal University, Beijing 100875, PR China
}
\email{yongjian@bnu.edu.cn}

\author{Xuzhou Zhan$^\ast$} 
\address{
Department of Mathematics, Beijing Normal University at Zhuhai,
    Zhuhai 519087, PR China
}
\email{xzzhan@bnu.edu.cn}
\thanks{$^\ast$Corresponding author}


\date{\today}


\begin{abstract}
We develop a Hurwitz stability criterion for nonmonic matrix polynomials via column reduction, generalizing existing approaches constrained by the monic assumption and thus serving as a more natural extension of Gantmacher’s classical stability criterion via Markov parameters.
Starting from redefining the associated Markov parameters through an adaptive ‌column-wise Li\'{e}nard-Chipart splitting method, our framework constructs two structured matrices whose rectangular Hankel blocks are obtained via the extraction of these parameters. 
We establish an explicit interrelation between the inertia of column reduced matrix polynomials and the derived structured matrices. Furthermore, we demonstrate that the Hurwitz stability of column reduced matrix polynomials can be determined by the Hermitian positive definiteness of these rectangular block Hankel matrices. 
\end{abstract}

\maketitle

\section{Introduction}\label{SectionIntro}

Consider a matrix polynomial 
\begin{equation}\label{RepresentF} 
F(\lambda)=A_{0} \lambda^{n}+A_{1} \lambda^{n-1}+\cdots+A_{n},\quad A_0\neq 0_p,
\end{equation}
where $A_0,\ldots, A_n$ are $p\times p$ complex matrices and called the coefficients of $F(\lambda)$. Particularly, $A_0$ is called the leading coefficient of $F(\lambda)$, and $n$ is called the degree of $F(\lambda)$, denoted by $\deg F(\lambda)$. We call $F(\lambda)$ to be monic 
if $A_0=I_p$. Assume that $F(\lambda)$ is regular, i.e., $\det F(\lambda)$ is not identically zero. 
Recall that a complex
value $\mu$ is called a finite eigenvalue 
of $F(\lambda)$ if $\det F(\mu)=0$.  
The localization of eigenvalues within typical regions (e.g. the unit disk  \cite{DyV,LT82,SD}, the major half-planes \cite{CIK,NN,MMW,ZBH}, and the sectors \cite{Mat,MMW}) plays a critical role in analyzing the stability of differential/difference systems, as well as the performance of signal processing.
The spectral inclusion regions are crucial as well for the performance in the shift-and-invert methods and the error analysis for  the polynomial eigenvalue solvers (see, e.g., \cite{BH,RB,KG} and the references therein).

The eigenvalue localization problem considered in this paper is to determine whether a regular matrix polynomial $F(\lambda)$ is Hurwitz stable, i.e., all finite eigenvalues of $F(\lambda)$ are located in the open left half-plane. The Hurwitz stability of $F(\lambda)$ in \eqref{RepresentF} directly characterizes the asymptotic stability of the associated high-order differential system
\begin{equation*}
A_{0} y^{(n)}(t)+A_{1} y^{(n-1)}(t)+\cdots+A_{n} y(t)=u(t),
\end{equation*} 
where $y(t)$ and $u(t)$ denote the output vector and the control input vector, respectively \cite{GLRMP}.

For scalar polynomials with $p=1$, the stability criteria can be broadly divided into two classes  (see also \cite{ALCRMD,HT}). The first class comprises the decisive algebraic tests directly via polynomial coefficients: the Routh–Hurwitz algorithm, and the Liénard–Chipart refinement  \cite{BAK,HP}, the Bezoutian/Hermite–Fujiwara matrix  \cite{Bar}, the Markov-parameter approach  \cite{Gan}, together with recent developments on its numerical implementation for linear time-invariant systems \cite{FE} and on the topology of the space of Markov parameters \cite{LAF}. 
 The second class comprises the complex-analytic criteria based on the intuitive geometric features of polynomials: argument principle such as Leonhard–Mihailov criterion and Nyquist plot  \cite{Leo}, rational function mapping such as Hermite–Biehler theorem  \cite{Gan} and Mobius transform  \cite{Aguirre,ADLS}. In the first-degree case that $F(\lambda)=I_p \lambda-A$ corresponding to linear time-invariant systems, its Hurwitz stability is equivalent to the solvability of the Lyapunov equation with a positive definite solution  \cite{Kush,Lyapu}, and recent criteria also include approaches through matrix functions, such as the matrix sign function  \cite{FD} and the matrix exponential function  \cite{FF}. 

Recent decades have witnessed the extension of Hurwitz stability analysis to matrix polynomials of general degree, generalizing methods from the scalar and first-degree cases: structured algebraic constructions \cite{CM,LRT,LT82,XD,Gal}, LMI optimization \cite{HAP,HAPS,LPJ} as the extended Lyapunov approach, the argument principle \cite{BA,Hu}, the Cayley transform \cite{RB} relating Hurwitz and Schur stability, and a partial extension of the Hermite–Biehler theorem \cite{ZH}. However, limited by the noncommutativity of the matrix coefficients, these extended tools still face difficulties in testing the Hurwitz stability of matrix polynomials. For instance, the well-definedness, self-adjointness, or large dimensions arising in the construction of several block structured matrices (e.g., Anderson-Jury B\'{e}zoutian matrix \cite{AnJu, LRT,LT82}, block Hankel matrices \cite{XD,ZBH}) pose a challenge to computation and practical application.

Associated with a particular well-defined class of Anderson-Jury B\'{e}zoutian matrices, two Hermitian block Hankel matrices from Hermitian Markov parameters of matrix polynomials are constructed by \cite{XD,ZBH} to test the stability of matrix polynomials with their definiteness (see Hurwitz stability tests \cite[Theorems~4.4 and 4.10]{XD} and quasi-stability test \cite[Theorem~3.2]{ZBH}).  
In particular, the Hurwitz stability conditions in \cite{XD} are algebraically characterized in relationship to the Hermitian positive definiteness of two block Hankel matrices.
This provides a matricial 
refinement of Gantmacher's classical stability criterion via Markov parameters \cite[Theorem 17, Chapter XV]{Gan}.  
However, the framework proposed in \cite{ZBH,XD} imposes a crucial restriction: the matrix polynomial under consideration must be monic (or more generally, has a nonsingular leading coefficient) to ensure the well-posedness of its Markov parameters. 

This paper extends the aforementioned relationship in \cite{XD} to nonmonic matrix polynomials (whose leading coefficients are not necessarily nonsingular)--structures that naturally arise in differential-algebraic systems and inherently encode their algebraic constraints (see, e.g., \cite{KM}). 
The column reduced matrix polynomials (see the definition in Section \ref{Sec2}), covering canonical forms such as the Popov normal form \cite{Popov} and Hermite normal form \cite{Hermite}, constitute a fundamental class of regular matrix polynomials.
Each regular matrix polynomial can be transformed via column (row) reduction into a column (row) reduced matrix polynomial without variation of its eigenvalues,
yielding
a polynomial basis with minimal column (or row) degrees for the free $\mathbb C[\lambda]$-module generalized by its columns (or rows). We refer the reader to \cite{GJV,LNZ,NRS} and the reference therein for details on fast randomized or deterministic algorithms for column (row) reduction. 

In this paper,
the framework of \cite{XD,ZBH} is substantially modified to remove the monic restriction. 
Our new strategy starts from extending the foundational framework by redefining‌ matricial Markov parameters‌ for $F(\lambda)$‌.  Note that the Li\'{e}nard-Chipart splitting is essential in the construction of the Markov parameters in \cite{XD}, which divides the original matrix polynomial $F(\lambda)$ into the even-odd parts. However, this splitting fails to produce the desired Markov parameters in nonmonic settings. 
To resolve this, we propose an adaptive ‌column-wise Li\'{e}nard-Chipart splitting method, which operates in two steps. Specifically, for a nonmonic matrix polynomial $F(\lambda)$ in column reduced form (obtained, if necessary, by column reduction preprocessing),

\ref{step1}: Apply the Li\'{e}nard-Chipart splitting independently to each column of 
$F(\lambda)$.

\ref{step2}: Merge the resulting even-odd components to form the columns of the so-called dominant-subdominant pair of matrix polynomials $(F_d(\lambda), F_s(\lambda))$.

This splitting preserves several key advantages of the original splitting in \cite{XD}. In particular, the polynomial degrees in the dominant-subdominant pair are reduced by at least half, and the dominant part $F_d(\lambda)$ maintains a column-reduced structure. These properties ensure the well-definedness of the rational matrix (i.e., the matrix whose
entries are rational functions) $F_s(\lambda)(F_d(\lambda))^{-1}$ and the associated sequence of Markov parameters $\mathscr S$ in \ref{step3} and, further, facilitate a recursive construction of $\mathscr S$ in Algorithm \ref{alg1}.
In this regard,
two rectangular block Hankel matrices can be built
via two steps.

\ref{step4}: Extract some submatrices from each Markov parameter by retaining rows and columns indexed by a particular column set of the dominant part $F_d(\lambda)$.

\ref{step5}: Construct the structured matrix $\mathbf{H}_{0}^{\mathscr S}$ and the shifted counterpart $\mathbf{H}_{1}^{\mathscr S}$, where the compressed Markov parameters constitute the rectangular Hankel blocks.  

Following \cite{XD}, we require $\mathscr S$ to be Hermitian, which may not be satisfied for all regular matrix polynomials. This property naturally holds for the classical Markov parameters of scalar real polynomials. Further, without this condition, the connection between $\mathscr S$ and the Hurwitz stability of $F(\lambda)$,  as established through Gantmacher's theorem in the scalar case, is not expected in general to carry over to matrix polynomials (see Claims 2.4-2.5 in \cite{XD}). 
We also remark that this property
may be ensured through appropriate feedback control design (see Remark \ref{RemFeedback} and Example \ref{DAE3order}). A further importance of this condition is that, to bridge the Hurwitz stability of $F(\lambda)$ with $\mathbf{H}_{0}^{\mathscr S}$ and $\mathbf{H}_{1}^{\mathscr S}$, it allows us to fully utilize the finite Hankel and B\'ezoutian matrices associated with the column reduced matrix polynomials, drawing foundational insights from the work of Van Barel et al.
\cite{BPV} and the references therein.

By exploiting the structure of the finite Hankel and finite B\'ezoutian
matrices, we replace an associated Anderson--Jury
B\'ezoutian matrix, which is of large dimension and involves four matrix
polynomials in its construction, by two structured
matrices $\mathbf H_0^{\mathscr S}$ and $\mathbf H_1^{\mathscr S}$ of
smaller sizes, thereby establishing an explicit representation of the
inertia of $F(\lambda)$ (see Theorem~\ref{ThmInertia}).
Moreover, we demonstrate that the Hurwitz stability of $F(\lambda)$ can be determined by the Hermitian positive definiteness of $\mathbf{H}_{0}^{\mathscr S}$ and $\mathbf{H}_{1}^{\mathscr S}$ (see Theorem~\ref{ThmHurwitz}). These results reduce to the inertia representation  \cite[Theorem 3.5]{XD} and the Hurwitz stability criterion \cite[Theorem 4.4]{XD} when $F(\lambda)$ is monic, thus developing a more natural, column-reduced extension of Gantmacher's classical stability criterion via Markov parameters (see Remarks \ref{RemReduction} and \ref{RemReduction1}). This criterion requires neither linearization nor direct determinant or
eigenvalue computations.
Compared with the treatment \cite{XD} under the restrictive monic assumption, the determining  Hermitian  matrices $\mathbf{H}_{0}^{\mathscr S}$ and  $\mathbf{H}_{1}^{\mathscr S}$ are compressed into the form of rectangular Hankel blocks and their sizes can even be reduced in the column reduced setting. 
 Our stability criterion can be readily adapted to row reduced matrix polynomials through an analogous construction, although we omit the details here for brevity.

This paper is organized as follows. Section \ref{Sec2} introduces some notation and concepts on matrix polynomials. In Section \ref{Sec3} we extend some properties of finite Hankel matrices and finite B\'ezoutian matrices. Section \ref{sec4} is the main part of the whole paper, which is divided into two subsections. In Subsection \ref{subsec4.1}, we construct the sequence of Markov parameters for column-reduced matrix polynomials. Subsection \ref{subsec4.2} establishes an explicit  representation for the inertia of column reduced matrix polynomials and derives Hurwitz stability criteria for such polynomials. Illustrative numerical examples are also provided in this section. The definition of the extended infinite companion matrix used in Section \ref{Sec3} is given in \ref{SecApp}.

\section{Preliminaries}\label{Sec2}
 Let $\C$, $\R$, $\N_0$ and $\N$ be the sets
of all complex, real, nonnegative integer, and positive integer numbers, respectively. We denote by $\mathbb C^{p\times q}$ the set of all $p\times q$ complex matrices, and by  $I_p$ the $p\times p$ identity matrix. 
Given a matrix $A$, we denote by $\bar A$ and $A^*$ its conjugate matrix and  conjugate transpose matrix respectively. If $A\in \mathbb C^{p\times p}$ is a Hermitian matrix, we write $A\succ 0$ if it is positive definite, and $A\succeq 0$ if $A$ is positive semidefinite.
Given a matrix $A\in \mathbb C^{p\times q}$ and two index sets $\mathcal I_r\subseteq \{1,\ldots,p\}$ and $\mathcal I_c\subseteq \{1,\ldots,q\}$, let $A_{\mathcal I_r}$  and $A^{\mathcal I_c}$ denote the submatrix of $A$ formed by rows of $A$ with indices belonging to $\mathcal I_r$ (in increasing order) and the submatrix of $A$ formed by columns of $A$ with indices belonging to $\mathcal I_c$ (in increasing order), respectively. For simplicity we write $A_{\mathcal I_r}^{\mathcal I_c}$ for $(A_{\mathcal I_r})^{\mathcal I_c}$.

We use the symbol $\mathbb C[\lambda]^{p\times q}$ to stand for the set of $p\times q$ matrix polynomials, which is the set of matrix polynomials in $\lambda$ with coefficients from $\mathbb C^{p\times q}$. In particular, $\mathbb C[\lambda]=\mathbb C[\lambda]^{1\times 1}$. 
Given a matrix polynomial~$F(\lambda)\in\mathbb C[\lambda]^{p\times p}$ written as
in~\eqref{RepresentF}, define $F^{\vee}(\lambda),\ F^{T}(\lambda),\ \overline{F}(\lambda)\in\mathbb C[\lambda]^{p\times p}$ by
\begin{equation*}
F^{\vee}(\lambda)\coloneqq\sum_{k=0}^n A^*_k \lambda^{n-k},\quad F^{T}(\lambda)\coloneqq\sum_{k=0}^n A^T_k \lambda^{n-k},\quad \overline{F}(\lambda)\coloneqq\sum_{k=0}^n \overline{A_k} \lambda^{n-k}.
\end{equation*}
Obviously, $F^{\vee}(\lambda)=\overline{F^{T}}(\lambda)=(\overline{F})^{T}(\lambda)$.
Suppose that $F(\lambda), L(\lambda)\in \mathbb C[\lambda]^{p\times p}$. $L(\lambda)$ is called a right divisor
of $F(\lambda)$ 
if there exists an
    $M(\lambda)\in \mathbb C[\lambda]^{p\times p}$ such that
    \begin{equation*}
        F(\lambda)=M(\lambda)L(\lambda).
    \end{equation*}
    Let additionally $\wtilde F(\lambda)\in \mathbb C[\lambda]^{p\times p}$. Then~$L(\lambda)$ is called a common right divisor 
of~$F(\lambda)$ and $\wtilde F(\lambda)$ if
    $L(\lambda)$ is a right  divisor of $F(\lambda)$ and also a right  divisor of
    $\wtilde F(\lambda)$.
  Furthermore, $L(\lambda)$ is called a greatest
    common right divisor (GCRD) of $F(\lambda)$ and $\wtilde F(\lambda)$
    if any other  common right divisor is a right divisor of $L(\lambda)$. In particular, $F(\lambda)$ and $\wtilde F(\lambda)$ are said to be right coprime if any  
   GCRD of $F(\lambda)$ and $\wtilde F(\lambda)$ is unimodular, that is, its determinant is a nonzero constant \cite{NV,Kai}.

 Recalling the definition of the inertia of a matrix: For~$A\in \mathbb C^{p\times p}$, the
inertia of~$A$ with respect to the imaginary axis~${\rm i}\R$ is defined by the
triple \label{PagInertia}
\begin{equation*}
\In(A)\coloneqq(\pi(A),\nu(A),\delta(A)),
\end{equation*}
where~$\pi(A)$, $\nu(A)$, and~$\delta(A)$ stand for the number of eigenvalues (counting
algebraic multiplicities) of~$A$ with positive, negative, and zero real parts, respectively.

For the inertia of regular matrix polynomials in the  complex plane, let us adopt the notation from~\cite{LRT}, which is
essentially as in~\cite{LT82} (see also~\cite[Proposition 2.2]{LRT}).
    Let~$F(\lambda)\in {\mathbb C}[\lambda]^{p\times p}$ be regular. For a finite eigenvalue $\mu$ of $F(\lambda)$,
its algebraic multiplicity is the multiplicity of $\mu$ as a zero of $\det F(\lambda)$. 
The spectrum of $F(\lambda)$, denoted by $\sigma(F)$, is the set of all finite eigenvalues of $F(\lambda)$ in $\mathbb C$. The eigenvalues of matrix polynomials can also be defined in the extended complex plane. Let $F(\lambda)$ be a nonzero matrix polynomial of degree $n$.
The reversal matrix polynomial of $F(\lambda)$ is defined as follows:
  $$
  {Rev}[F](\lambda)=\lambda^n F(\lambda^{-1}).
  $$
  Recall that $\infty$ is called an eigenvalue of $F(\lambda)$ if $0$ is an eigenvalue of ${Rev}[F](\lambda)$. Further, its algebraic multiplicity, denoted by $\gamma_{\infty}(F)$, is the algebraic multiplicity of $0$ as an eigenvalue of ${Rev}[F](\lambda)$. Denote by~$\gamma_+(F)$, $\gamma_-(F)$, $\gamma_{0}(F)$ the number of
    finite eigenvalues of~$F(\lambda)$ (counting with algebraic multiplicities), in the open upper half plane, the
    open lower half plane and on the real line (excluding infinity), respectively. The triple
\begin{equation*}
\gamma(F)\coloneqq(\gamma_+(F),\gamma_-(F),\gamma_{0}(F))
\end{equation*}
 is called 
the inertia
of~$F(\lambda)$ with respect to~$\R$.
Analogously, the triple
\begin{equation*}
\gamma'(F)\coloneqq(\gamma'_+(F),\gamma'_-(F),\gamma'_{0}(F))
\end{equation*}
is called the inertia of~$F(\lambda)$
with respect to~${\rm i}\R$, replacing the upper half plane by the right half
plane, the lower half plane by the left half plane, and the real axis (excluding infinity)
by the imaginary axis ${\rm i}\mathbb R$ (excluding infinity), respectively.  It is obvious that 
\begin{equation*}
\deg\det F=\gamma'_+(F)+\gamma'_-(F)+\gamma'_{0}(F)=\gamma_+(F)+\gamma_-(F)+\gamma_{0}(F).
\end{equation*}

   For a column vector polynomial $f(\lambda)\in \mathbb{C}[\lambda]^{p\times 1}$, we  denote by $\deg f$ the degree of  $f(\lambda)$, that is, the highest degree of all entries of $f(\lambda)$.
Let $F(\lambda)\in \mathbb C[\lambda]^{p\times p}$ be regular and partitioned into $p$ columns
\begin{equation}\label{Partition}
F(\lambda)=\begin{bmatrix} f_1(\lambda) & \cdots & f_p(\lambda)\end{bmatrix},
\end{equation}
where $f_k(\lambda)\in \mathbb C[\lambda]^{p\times 1}$. 
Then $({\deg\, f_k})_{k=1}^p$ is called the sequence of column degrees of $F(\lambda)$ and $\deg\, f_k$ is denoted by ${\rm cdeg}_k F$ for $k\in \{1,\ldots,p\}$. Further, $F(\lambda)$ can be written as 
$$
F(\lambda)=A\cdot {\rm diag}(\lambda^{{\rm cdeg}_1 F},\ldots, \lambda^{{\rm cdeg}_p F})+L(\lambda),
$$
where ${\rm cdeg}_k L<{\rm cdeg}_k F$ for $k=1,\ldots,p$.
Here $A$ is called the highest column degree coefficient matrix of $F(\lambda)$ and denoted by $F_{\mathrm{hcdc}}$.
In particular,
$F(\lambda)$ is called column reduced if $F_{\mathrm{hcdc}}$ is nonsingular.

The concept of row reduced matrix polynomials is defined as the row analogue to column reduced matrix polynomials. For a row vector polynomial $f(\lambda)\in \mathbb{C}[\lambda]^{1\times p}$, we  denote by deg $f$ the degree of  $f(\lambda)$, that is, the highest degree of all entries of $f(\lambda)$. Let $F(\lambda)\in \mathbb C[\lambda]^{p\times p}$ be regular and partitioned into $p$ rows
\begin{equation*}
F(\lambda)=\begin{bmatrix}
f_1(\lambda)\\
\vdots\\
f_p(\lambda)
\end{bmatrix}
\end{equation*}
where $f_k(\lambda)\in \mathbb C[\lambda]^{1\times p}$. 
Then $({\deg\, f_k})_{k=1}^p$ is called the sequence of row degrees of $F(\lambda)$ and $\deg\, f_k$ is denoted by ${\rm rdeg}_k F$ for $k\in \{1,\ldots,p\}$. Further, $F(\lambda)$ can be written as
$$
F(\lambda)={\rm diag}(\lambda^{{\rm rdeg}_1 F},\ldots,\lambda^{{\rm rdeg}_p F})\cdot A+L(\lambda),
$$
where ${\rm rdeg}_k L<{\rm rdeg}_k F$ for $k=1,\ldots,p$. Here $A$ is called the highest row degree coefficient matrix of $F(\lambda)$ and denoted by $F_{\mathrm{hrdc}}$.
In particular,
$F(\lambda)$ is called row reduced if $F_{\mathrm{hrdc}}$ is nonsingular.
Obviously, $F(\lambda)$ is row reduced if and only if $F^{T}(\lambda)$ is column reduced, or equivalently, $F^{\vee}(\lambda)$ is column reduced.

 The following notations are needed to label the rows/columns of infinite Hankel matrices and B\'ezoutian matrices introduced in Section \ref{Sec3}.
Given a regular matrix polynomial $F(\lambda)\in \mathbb C[\lambda]^{p\times p}$, its row index set (see \cite{BPV}) is defined as
$$
\mathcal I_r(F):=\bigcup_{i=1}^p[i]^r_{F},
$$
where 
$$
[i]^r_{F}:=\begin{cases}
\emptyset, & {\rm if }\ {\rm rdeg}_iF=0;\\
\{t:t=jp+i,0\leq j< {\rm rdeg}_iF\}, & {\rm if }\ {\rm rdeg}_iF>0.
\end{cases}
$$  
It is natural for us to define the column counterpart, that is, the column index set  of $F(\lambda)$ is given as
$$
\mathcal I_c(F):=\bigcup_{i=1}^p[i]^c_{F},
$$
where 
$$
[i]^c_{F}:=\begin{cases}
\emptyset, & {\rm if }\ {\rm cdeg}_iF=0;\\
\{t:t=jp+i,0\leq j< {\rm cdeg}_iF\}, & {\rm if }\ {\rm cdeg}_iF>0.
\end{cases}
$$
Obviously, there is a connection between the above two notations:
\begin{equation}\label{Icr}
\mathcal I_c(F)=\mathcal I_{r}(F^{\vee})=\mathcal I_{r}(F^{T}).
\end{equation}

\section{Finite Hankel matrices and B\'ezoutian matrices}\label{Sec3}

For our stability analysis in Section \ref{sec4}, this section extends some properties of finite Hankel matrices and B\'ezoutian matrices associated with column/row reduced matrix polynomials.

We denote by $\mathbb C(\lambda)$ the field of rational functions over $\mathbb C[\lambda]$, and by $\mathbb C(\lambda)^{p\times q}$ the set of $p\times q$ rational matrices whose  elements belong to $\mathbb C(\lambda)$. Obviously $\mathbb C[\lambda]^{p\times q}\subseteq \mathbb C(\lambda)^{p\times q}$.
Suppose that $R(\lambda)\in \mathbb C(\lambda)^{p\times p}$.  $R(\lambda)$ is called proper (resp. strictly proper) if 
$
\lim \limits_{\lambda\rightarrow \infty} R(\lambda)\ \mbox{is a constant matrix}$
(resp.$\ \lim \limits_{\lambda\rightarrow \infty} R(\lambda)=0_p).
$
In the proper case, assume that $R(\lambda)$ admits the Laurent series
\begin{equation}\label{LaurentSeriesRsp}
       R(\lambda)={\bf s}_{-1}+\sum_{j=0}^{\infty} \frac{{\bf s}_j}{\lambda^{j+1}},\quad \lambda\rightarrow \infty.
     \end{equation}
Define an infinite block Hankel matrix associated with~$R(\lambda)$ as
$$
\mathscr{H}_{\infty}(R):=({\bf s}_{i+j})_{i,j=0}^{\infty}.
$$
Suppose that $R(\lambda)$ admits both left and right matrix fraction descriptions
\begin{equation}\label{Matrixfrac}
R(\lambda)=(D_l(\lambda))^{-1}N_l(\lambda)=N_r(\lambda)(D_r(\lambda))^{-1},
\end{equation}
where $\{D_l(\lambda), N_l(\lambda), D_r(\lambda), N_r(\lambda)\}\subseteq \mathbb C[\lambda]^{p\times p}$, $D_l(\lambda)$ is column reduced and $D_r(\lambda)$ is row reduced. 
In 2001, Van Barel et al. \cite{BPV} construct the so-called finite Hankel matrix corresponding to the matrix fraction decomposition \eqref{Matrixfrac} of $R(\lambda)$, that is,
a submatrix of $\mathscr{H}_{\infty}(R)$ given via 
$$
\mathscr{H}(D_l,N_l;N_r,D_r):=\mathscr{H}_{\infty}(R)^{\mathcal I_r(D_r)}_{\mathcal I_c(D_l)}\in \mathbb C^{\deg \det D_l\times \deg \det D_r}.
$$ 

\cite[Theorem 50]{BPV1999} demonstrates that, for a strictly proper rational matrix $R(\lambda)$ with the matrix fraction descriptions \eqref{Matrixfrac},
$\mathscr{H}_{\infty}(R)$ has a factorization formula in terms of $\mathscr{H}(D_l,N_l;N_r,D_r)$ and two infinite  companion matrices (see the definition in \ref{SecApp}). 
The following result is an immediate extension of \cite[Theorem 50]{BPV1999} for the proper case. 
\begin{proposition}\label{ProHinfty}
Let $R(\lambda)\in \mathbb C(\lambda)^{p\times p}$ be a proper rational matrix with the matrix fraction decomposition \eqref{Matrixfrac}, where $\{D_l(\lambda), N_l(\lambda), D_r(\lambda), N_r(\lambda)\}\subseteq \mathbb C[\lambda]^{p\times p}$, $D_l(\lambda)$ is column reduced and $D_r(\lambda)$ is row reduced. 
Then 
$$
\mathscr{H}_{\infty}(R)=C_{\infty}(D_l^T)^T \mathscr{H}(D_l,N_l;N_r,D_r) C_{\infty}(D_r). 
$$
\end{proposition}

 Let $F(\lambda)\in \mathbb{C}[\lambda]^{p\times p}$ be regular. Due to \cite[Lemma 14 (4)]{BPV1999}, there exists a unimodular matrix polynomial $U(\lambda)$ such that $$F(\lambda)U(\lambda)=P(\lambda),$$ where $P(\lambda)$ satisfies the following properties:
\begin{enumerate}
    \item[{\rm (i)}] $P(\lambda)$ is column reduced and $P_{\mathrm{hcdc}}$ is unit upper triangular;
    \item[{\rm (ii)}] $P(\lambda)$ is row reduced and $P_{\mathrm{hrdc}}=I_p$.
\end{enumerate} 
In this case, $P(\lambda)$ is called the (unique) Popov canonical form (see, e.g., \cite[Definition 15]{BPV1999}) of $F(\lambda)$. Denote $P(\lambda)$ and $U(\lambda)$ by $\mathscr {P}[F](\lambda)$ and $\mathscr {U}[F](\lambda)$, respectively. We provide the following property for finite Hankel matrices associated with matrix polynomials and their Popov canonical forms.
 
\begin{proposition}\label{ProHH*}
Let $D(\lambda)\in \mathbb C[\lambda]^{p\times p}$ be column reduced. Then $\mathscr{H}(D,I_p;\mathscr U[D],\mathscr P[D])$ is nonsingular and 
\begin{equation}\label{eqHH*}
\mathscr{H}(D,I_p;\mathscr U[D],\mathscr P[D])=\mathscr{H}(\mathscr P[D]^{\vee},\mathscr U[D]^{\vee};I_p,D^{\vee})^*.
\end{equation}
\end{proposition}

\begin{proof}
Observing the right coprimeness of $\mathscr U[D]$ and $\mathscr P[D]$ and that of $I_p$ and $D^{\vee}(\lambda)$, together with \cite[Remark 15]{BPV}, we obtain that $\mathscr{H}(D,I_p;\mathscr U[D],\mathscr P[D])$ is nonsingular. 
By \eqref{Icr}, we have
\begin{equation}\label{eq1HH*}
\mathscr{H}(D,I_p;\mathscr U[D],\mathscr P[D])=\mathscr{H}_{\infty}(D^{-1})^{\mathcal I_r(\mathscr P[D])}_{\mathcal I_c(D)}=\mathscr{H}_{\infty}(D^{-1})^{\mathcal I_c(\mathscr P[D]^{\vee})}_{\mathcal I_r(D^{\vee})}.
\end{equation}
Note that each block element of $\mathscr{H}_{\infty}(D^{-1})$ is the conjugate transpose of the corresponding one of  $\mathscr{H}_{\infty}(D^{-\vee})$. Then we derive that
\begin{equation}\label{eq2HH*}
\mathscr{H}_{\infty}(D^{-1})^{\mathcal I_c(\mathscr P[D]^{\vee})}_{\mathcal I_r(D^{\vee})}
=\left(\mathscr{H}_{\infty}(D^{-\vee})^{\mathcal I_r(D^{\vee})}_{\mathcal I_c(\mathscr P[D]^{\vee})}\right)^*=\mathscr{H}(\mathscr P[D]^{\vee},\mathscr U[D]^{\vee};I_p,D^{\vee})^*.
\end{equation}
Hence, \eqref{eqHH*} follows from a combination of \eqref{eq1HH*} and \eqref{eq2HH*}.
\end{proof}

Suppose that $R(\lambda)$ is proper and admits the left and right matrix fraction descriptions \eqref{Matrixfrac},
where $\{{D}_l(\lambda), {N}_l(\lambda)$, $D_r(\lambda), {N}_r(\lambda)\}\subseteq \mathbb C[\lambda]^{p\times p}$, ${D}_l(\lambda)$ is row reduced and ${D}_r(\lambda)$ is column reduced.
In this case, it is natural for us to define an alternative finite Hankel matrix as the counterpart given via 
$$
{\mathbf{H}}(D_l,N_l;N_r,D_r):=\mathscr{H}_{\infty}(R)^{\mathcal I_r(D_l)}_{\mathcal I_c(D_r)}\in \mathbb C^{\deg \det  D_r\times \deg \det D_l}.
$$

A rational matrix $R(\lambda)\in \mathbb C(\lambda)^{p\times p}$ is called symmetric with respect to $\mathbb R$ if it obeys that $R(t)=R(\bar t)^*$ for all $t\in\mathbb C$ except the poles of the entries of $R(\lambda)$ (see more details in \cite{DNQD}). 
In what follows, we bridge a connection between two types of finite Hankel matrices related to these typical rational matrices.

\begin{proposition}\label{ProHscrH}
Let $R(\lambda)\in \mathbb C(\lambda)^{p\times p}$ be proper and symmetric with respect to $\mathbb R$ with the matrix fraction descriptions 
\begin{equation}\label{SymMatrixfrac}
R(\lambda)=(D_r^{\vee}(\lambda))^{-1}N_r^{\vee}(\lambda)=N_r(\lambda)(D_r(\lambda))^{-1},
\end{equation}
where $\{D_r(\lambda), N_r(\lambda)\}\subseteq \mathbb C[\lambda]^{p\times p}$, and $(D_r)_{\mathrm{hcdc}}$ is nonsingular upper triangular.
Then
\begin{equation*}
\mathbf{H}(D^{\vee}_r,N^{\vee}_r;N_r,D_r)=\mathscr{H}(\mathscr P[D_r]^{\vee},\mathscr U[D_r]^{\vee} N_r^{\vee};N_r \mathscr U[D_r],\mathscr P[D_r]).
\end{equation*}
\end{proposition}

\begin{proof}

Analogous to the proof of \cite[Lemma 14 (2)]{BPV1999}, one can get that
\begin{align*}
\mathscr{U}[D_r](\lambda)=&D_r(\lambda)^{-1} \mathscr{P}[D_r](\lambda)\\
=&\mathrm{diag}(\lambda^{-\mathrm{cdeg}_1D_r},\ldots,\lambda^{-\mathrm{cdeg}_pD_r})(A+\wtilde R(\lambda))\\
&\cdot \mathrm{diag}(\lambda^{\mathrm{cdeg}_1\mathscr{P}[D_r]},\ldots,\lambda^{\mathrm{cdeg}_p\mathscr{P}[D_r]}),
\end{align*}
where $A:=((D_r)_{\mathrm{hcdc}})^{-1}\mathscr{P}[D_r]_{\mathrm{hcdc}}$ and $\wtilde R(\lambda)\in \mathbb C(\lambda)^{p\times p}$ is strictly proper.
From the nonzero diagonal entries of $A$ and the column reducedness of $D_r(\lambda)$ and $\mathscr{P}[D_r](\lambda)$, we deduce that 
\begin{equation*}
\mathrm{cdeg}_k \mathscr{P}[D_r] = \mathrm{cdeg}_k D_r, \quad k=1,\ldots,p.
\end{equation*}
Thus,
\begin{equation}\label{IcDF}
\mathcal{I}_c(D_r)=\mathcal{I}_c(\mathscr{P}[D_r]).
\end{equation}
It is clear that 
$$
\mathrm{cdeg}_k (\mathscr {P}[D_r])=\mathrm{rdeg}_k (\mathscr {P}[D_r]),
\quad k=1,\ldots,p,
$$
which shows that
\begin{equation}\label{IcF}
\mathcal{I}_c(\mathscr {P}[D_r])=\mathcal{I}_r(\mathscr {P}[D_r]).
\end{equation}
Hence, combining \eqref{Icr}, \eqref{IcDF} and \eqref{IcF} yields that
$$
\mathbf{H}(D^{\vee}_r, N^{\vee}_r; N_r, D_r) 
    = \mathscr{H}_{\infty}(R)^{\mathcal{I}_r(\mathscr{P}[D_r])}_{\mathcal{I}_c(\mathscr{P}[D_r]^{\vee})}
    = \mathscr{H}\bigl(\mathscr{P}[D_r]^{\vee}, \mathscr{U}[D_r]^{\vee} N_r^{\vee}; 
        N_r \mathscr{U}[D_r], \mathscr{P}[D_r]\bigr).
$$

\end{proof}

Now we dig slightly deeper to seek interior block structure of  finite Hankel matrices.
Let $F(\lambda)\in \mathbb C[\lambda]^{p\times p}$.
For $i=-1,\ldots,\deg F$, let $\mathcal I^r_i(F)$ denote the index set of all row indices for which the row degrees of $F(\lambda)$ exceed $i$
$$
\mathcal I^r_i(F):=\{k:i< {\rm rdeg}_kF\leq \deg F\}\subseteq \{1,\ldots,p\}.
$$
Similarly, define $
\mathcal I^c_i(F)$ as 
$$
\mathcal I^c_i(F):=\{k:i< {\rm cdeg}_kF\leq \deg F\}\subseteq \{1,\ldots,p\}.
$$
The following clearly holds:
\begin{remark}\label{RemIndex}
Let $F(\lambda)\in \mathbb C[\lambda]^{p\times p}$. Then
\begin{enumerate}
\item[{\rm (i)}] $\mathcal I^c_{\deg F}(F)=\mathcal I^r_{\deg F}(F)=\emptyset.$
\item[{\rm (ii)}] For $i=-1,\ldots,\deg F-1$, $\mathcal I^c_i(F)=\mathcal I^r_i(F^{\vee})=\mathcal I^r_i(F^{T}).$
    \item[{\rm (iii)}] $\{{\mathcal I}_i^r\}_{i=-1}^{\deg F-1}$ constitutes a nested family of subsets:
    $$ \{k:{\rm rdeg}_kF=\deg F\}=\mathcal I^r_{\deg F-1}(F)\subseteq \mathcal I^r_{\deg F-2}(F) \subseteq \cdots \subseteq \mathcal I^r_{-1}(F)=\{1,\ldots,p\}.$$
    \item[{\rm (iv)}] $\{{\mathcal I}_i^c\}_{i=-1}^{\deg F-1}$ constitutes a nested family of subsets: $$ \{k:{\rm cdeg}_kF=\deg F\}=\mathcal I^c_{\deg F-1}(F)\subseteq \mathcal I^c_{\deg F-2}(F) \subseteq \cdots \subseteq \mathcal I^c_{-1}(F)=\{1,\ldots,p\}.$$
    \item[{\rm (v)}] If $F(\lambda)$ is monic, 
    $
\mathcal I^c_{i}(F)=\mathcal I^r_{i}(F)=\{1,\ldots,p\}$ for $i=-1,\ldots,\deg F-1.
    $
\end{enumerate}

\end{remark}

\begin{proposition}\label{LemHs}
Let $R(\lambda)\in \mathbb C(\lambda)^{p\times p}$ be a proper rational matrix that admits the Laurent series \eqref{LaurentSeriesRsp} and matrix fraction decomposition \eqref{Matrixfrac}, where $\{D_l(\lambda), N_l(\lambda),D_r(\lambda), N_r(\lambda)\}\subseteq \mathbb C[\lambda]^{p\times p}$, $D_l(\lambda)$ is row reduced, and $D_r(\lambda)$ is column reduced.
Then \begin{equation}\label{HR}
\mathbf{H}(D_l,N_l;N_r,D_r)=({\bf  s}_{i,j})_{i,j=0}^{\deg D_r-1,\deg D_l-1},
\end{equation}
where ${\bf  s}_{i,j}:=\left({\bf s}_{i+j}\right)_{\mathcal I^c_i(D_r)}^{\mathcal I^r_j(D_l)}\in \mathbb{C}^{|\mathcal I^c_i(D_r)|\times |\mathcal I^r_j(D_l)|}$.

\end{proposition} 

\begin{proof} Let $k\in \{1,\ldots,p\}$. If ${\rm rdeg}_kD_l=0$,  $[k]^r_{D_l}=\emptyset$. When ${\rm rdeg}_kD_l>0$,  $$[k]^r_{D_l}=\{k,p+k,\ldots,({\rm rdeg}_kD_l-1)p+k\}.$$ In this case,
the columns of $\mathscr{H}_{\infty}( R)^{[k]^r_{D_l}}$ consist of all $k$-th columns of the submatrices 
\begin{equation*}
\begin{bmatrix}
{\bf s}_{0}\\
{\bf s}_{1}\\
\vdots
\end{bmatrix},\begin{bmatrix}
{\bf s}_{1}\\
{\bf s}_{2}\\
\vdots
\end{bmatrix},\ldots,\begin{bmatrix}
{\bf s}_{{\rm rdeg}_kD_l-1}\\
{\bf s}_{{\rm rdeg}_kD_l}\\
\vdots
\end{bmatrix},
\end{equation*}
Note that $\deg D_l=\max \{{\rm rdeg}_kD_l:k=1,\ldots,p\}$.
To sum up, the columns of $\mathscr{H}_{\infty}(R)^{\mathcal I_r(D_l)}$ contain all $k$-th columns of 
$
\begin{bmatrix}
{\bf s}_{j}\\
{\bf s}_{j+1}\\
\vdots
\end{bmatrix},
$
where $k\in \mathcal I^r_j(D_l)$ and $j=0,\ldots,\deg D_l-1$. Analogously, the rows of $\mathscr{H}_{\infty}(R)_{\mathcal I_c(D_r)}$ contain all $k$-th rows of 
$
\begin{bmatrix}
{\bf s}_{i} &
{\bf s}_{i+1} &\cdots
\end{bmatrix},
$
where $k\in \mathcal I^c_i(D_r)$ and $i=0,\ldots,\deg D_r-1$. Thus, the equation \eqref{HR} holds.

\end{proof}

At the end of this section we recall the generalized B\'ezoutian matrix associated with matrix polynomials. 
Let $L(\lambda),\wtilde L(\lambda),M(\lambda), \wtilde M(\lambda)\in \mathbb C[\lambda]^{p\times p}$
satisfy
\begin{equation}\label{CommonMultiple}
    \wtilde M(\lambda)\wtilde L(\lambda)=M(\lambda)L(\lambda).
\end{equation}
The so-called  Anderson-Jury B\'ezoutian matrix ${\mathbf B}_{n_1,n_2}(\wtilde M,\wtilde L;M,L)\in \mathbb C^{n_1 p\times n_2 p}$
associated with the quadruple~$(\wtilde M,\wtilde L,M,L)$ is defined as follows:
\begin{equation*}   
\begin{bmatrix}I_p & \lambda I_p & \cdots & \lambda^{n_1-1}I_p \end{bmatrix}\cdot
    {\mathbf B}_{n_1,n_2}(\wtilde M,\wtilde L;M,L)\cdot
    \begin{bmatrix}I_p \\ \mu I_p \\ \vdots \\ \mu^{n_2-1}I_p \end{bmatrix}
    =\frac1{\lambda-\mu} \left(\wtilde M(\lambda)\wtilde L(\mu)-M(\lambda)L(\mu)\right), 
\end{equation*}
where $n_1\coloneqq\max\{\deg M, \deg \wtilde M\}$ and
$n_2\coloneqq\max\{\deg L, \deg \wtilde L\}$. 

For commuting $L(\lambda)$ and $\wtilde L(\lambda)$, i.e., when~$L(\lambda)\wtilde L(\lambda)=\wtilde L(\lambda)L(\lambda)$, it is natural to choose~$\wtilde M(\lambda)=L(\lambda)$ and~$M(\lambda)=\wtilde L(\lambda)$ such that \eqref{CommonMultiple} holds. For a nontrivial
choice of $\wtilde M(\lambda)$ and $M(\lambda)$ satisfying \eqref{CommonMultiple} in the general non-commutative case, we refer the reader to the
construction of the common multiples via spectral theory of matrix polynomials:
see~\cite[Theorem 9.11]{GLRMP} for the monic case and~\cite[Theorem 2.2]{GKLR} for the comonic
case.

In 2001, Van Barel et al. \cite{BPV} obtain a compact version of B\'ezoutian matrix from the Anderson-Jury B\'ezoutian matrix.
Let $L(\lambda),\wtilde L(\lambda),M(\lambda), \wtilde M(\lambda)\in \mathbb C[\lambda]^{p\times p}$ be a quadruple of matrix polynomials 
satisfying \eqref{CommonMultiple}.
In particular, when $\wtilde M(\lambda)$ is row reduced and $L(\lambda)$ is column reduced, all rows and columns of
${\mathbf B}_{n_1,n_2}(\wtilde M,\wtilde L;M,L)$ are zeros excluding the ones  labeled by the row index set of $\wtilde M$ and the column index set of $L$, respectively. In this case, 
the finite B\'ezoutian matrix ${\mathbf B}(\wtilde M,\wtilde L;M, L)$ is defined as a submatrix of ${\mathbf B}_{n_1,n_2}(\wtilde M,\wtilde L;M,L)$ from deleting the aforementioned zero rows  and columns:
$$
{\mathbf B}(\wtilde M,\wtilde L;M, L):={\mathbf B}_{n_1,n_2}(\wtilde M,\wtilde L;M, L)_{\mathcal I_r({\wtilde M})}^{\mathcal I_c({L})}\in \mathbb C^{m\times l},
$$
where $m=\deg \det \wtilde M$ and $l=\deg \det L$.

In the following we give a congruence relation between the finite B\'ezoutian matrices and finite Hankel matrices:

\begin{proposition}\label{LemBH}
Let $R(\lambda)\in \mathbb C(\lambda)^{p\times p}$ be symmetric with respect to $\mathbb R$ with the matrix fraction descriptions \eqref{SymMatrixfrac},
where $\{D_r(\lambda), N_r(\lambda)\}\subseteq \mathbb C[\lambda]^{p\times p}$, and $(D_r)_{\mathrm{hcdc}}$ is nonsingular upper triangular. Then
$
{\bf B}(D_r^{\vee},N_r;N_r^{\vee}, D_r)$
is congruent to
$ \mathbf{H}(D_r^{\vee},N_r^{\vee};N_r,D_r).
$

\end{proposition}

\begin{proof}

In view of \cite[Lemma 32]{BPV} (see also \cite[Theorem 31 (c) (ii)]{BPV}), we derive that
\begin{align*}
&\mathscr{H}(\mathscr P[D_r]^{\vee},\mathscr U[D_r]^{\vee} N_r^{\vee};N_r \mathscr U[D_r],\mathscr P[D_r])\\
=&\mathscr{H}(\mathscr P[D_r]^{\vee},\mathscr U[D_r]^{\vee} ;I_p,D_r^{\vee}) {\bf B}(D_r^{\vee},N_r;N_r^{\vee}, D_r) \mathscr{H}(D_r,I_p;\mathscr U[D_r],\mathscr P[D_r]).
\end{align*}
Therefore, Proposition \ref{LemBH} follows from a combination of  Propositions \ref{ProHH*}--\ref{ProHscrH}.

\end{proof}

\section{Hurwitz stability and Markov parameters of column reduced matrix polynomials}\label{sec4}

In this section, we derive a Hurwitz stability criterion for column reduced matrix polynomials in terms of the related Markov parameters. Our discussion will be divided into two subsections.

\subsection{Construction of Markov parameters and rectangular block Hankel matrices}\label{subsec4.1}

In this subsection, we exhibit an adaptive ‌column-wise Li\'{e}nard-Chipart splitting method to construct the Markov parameters of column reduced matrix polynomials and the associated rectangular block Hankel matrices. Let us start with an extension of the Li\'{e}nard-Chipart splitting to vector polynomials.

Let $f(\lambda)\in \mathbb{C}^{p\times 1}[\lambda]$ be written as 
\begin{equation*}
f(\lambda)=\alpha_{0} \lambda^{n}+\alpha_{1} \lambda^{n-1}+\cdots+\alpha_{n},
\end{equation*}
where $n:=\deg\,f$ and $\alpha_k\in \mathbb{C}^{p\times 1}$ for $k=0,\ldots,n$.
 Then $f(\lambda)$ can always be split into the even part $f_{e}(\lambda)$ and the odd part $f_{o}(\lambda)$ as
    \begin{equation*}
  f_{e}(\lambda)=\sum_{k=0}^\frac n2\alpha_{2k}\lambda^{\frac n2-k}
  \quad\text{and}\quad
  f_{o}(\lambda)=\sum_{k=1}^\frac n2 \alpha_{2k-1}\lambda^{\frac n2-k}
\end{equation*}
when $n$ is even, and
\begin{equation*}
    f_{e}(\lambda)=\sum_{k=0}^\frac {n-1}2 \alpha_{2k+1}\lambda^{\frac {n-1}2-k}
    \quad\text{and}\quad
    f_{o}(\lambda)=\sum_{k=0}^\frac {n-1}2 \alpha_{2k}\lambda^{\frac {n-1}2-k}
\end{equation*}
when $n$ is odd,
     so that
    $$
        f(\lambda)=f_{e}(\lambda^2)+\lambda f_{o}(\lambda^2).
    $$ Further, $\lambda f_o(\lambda)$ is called the shifted odd part of $f(\lambda)$.

Let $F(\lambda)\in \mathbb C[\lambda]^{p\times p}$ be column reduced with the column partition \eqref{Partition} of nonzero degree $2m$ or $2m+1$.  
For convenience of the subsequent development, we first perform a preprocessing step on $F(\lambda)$.
We replace $F(\lambda)$ by 
$Q^*F(\lambda)$ to obtain a nonsingular upper triangular $F_{\mathrm{hcdc}}$, where $F_{\mathrm{hcdc}}=QR$ is a QR decomposition. This preprocessing leaves the eigenvalues and the column degrees of $F(\lambda)$ invariant and its Hurwitz stability unaffected. 
Then, we perform the adaptive ‌column-wise Li\'{e}nard-Chipart splitting on $F(\lambda)$, which is divided into two steps.

\begin{enumerate}[label=\textbf{Step \arabic*}, start=1, align=left, leftmargin=*] 
\item \label{step1} We apply the Li\'{e}nard-Chipart splitting independently to each column $f_k(\lambda)$ of $F(\lambda)$, dividing into the dominant and subdominant parts. That is, 
\begin{enumerate}[label={(\roman*)}, align=left, leftmargin=*, nosep] 
    \item when $\deg f_k$ is even, the even part of $f_k(\lambda)$ is called the dominant part of $f_k(\lambda)$ and denoted by $d_k(\lambda)$.
    \item when $\deg f_k$ is even, the shifted odd part of $f_k(\lambda)$ is called the subdominant part of $f_k(\lambda)$ and denoted by $s_k(\lambda)$.
    \item when $\deg f_k$ is odd, the odd part of $f_k(\lambda)$ is called the dominant part of $f_k(\lambda)$ and denoted by $d_k(\lambda)$.
    \item when $\deg f_k$ is odd, the even part of $f_k(\lambda)$ is called the subdominant part of $f_k(\lambda)$ and denoted by $s_k(\lambda)$.
\end{enumerate}

\item \label{step2}  We reassemble all the dominant parts and the subdominant parts into a pair of 
matrix polynomials. The 
constructed pair of matrix polynomials is of the form
\begin{equation*}
F_d(\lambda):=\begin{bmatrix}
    d_1(\lambda)& \cdots& d_p(\lambda)
\end{bmatrix}\in \mathbb C[\lambda]^{p\times p},\quad F_s(\lambda):=\begin{bmatrix} s_1(\lambda)& \cdots& s_p(\lambda)\end{bmatrix}\in \mathbb C[\lambda]^{p\times p}.
\end{equation*}
Obviously, this polynomial pair $(F_d(\lambda), F_s(\lambda))$ is uniquely determined by $F(\lambda)$ via the following decomposition: 
\begin{equation*}
F(\lambda)=D(\lambda)+S(\lambda),
\end{equation*}
where
\begin{align}
& D(\lambda):=F_d(\lambda^2)\alpha(\lambda)\in \mathbb C[\lambda]^{p\times p},\quad S(\lambda):=\lambda^{-1} F_s(\lambda^2)\alpha(\lambda)\in \mathbb C[\lambda]^{p\times p}, \label{DS} \\
& \alpha(\lambda):={\rm diag}(a_1(\lambda),\ldots,a_p(\lambda)),
\quad  
a_k(\lambda):=\begin{cases}
1, & {\rm cdeg}_k F {\rm\ is\ even},\\
\lambda, & {\rm cdeg}_k F {\rm\ is\ odd}.
\end{cases} \label{alphaz}
\end{align}
Then $F_d(\lambda)$ and $F_s(\lambda)$ are called the dominant and subdominant parts, respectively, of $F(\lambda)$.  
\end{enumerate}

\begin{enumerate}[label=\textbf{Step \arabic*}, start=3, align=left]
\item \label{step3}
We construct the sequence of Markov parameters from the pair $(F_d(\lambda),F_s(\lambda))$ as follows. 
\end{enumerate}

  We begin with the  definition of Markov parameters.  Since $F(\lambda)$ is column reduced,    $F_d(\lambda)$ is column reduced as well from the relation $(F_d)_{\mathrm{hcdc}}=F_{\mathrm{hcdc}}$. 
We construct the following well-defined rational matrix
\begin{equation}\label{Ge}        R_F(\lambda)\coloneqq F_{s}(\lambda)(F_{d}(\lambda))^{-1}\in \mathbb C(\lambda)^{p\times p}.
\end{equation}  
Note that the degree of each column of $F_s(\lambda)$ is no greater than that of $F_d(\lambda)$. 
According to \cite[Lemma 6.3-11]{Kai}, $R_F(\lambda)$ is proper and hence
    admits the Laurent expansion
    \begin{equation}
R_F(\lambda)=\sum_{k=0}^{\infty} (-1)^{k}\lambda^{-k}{\bf s}_k,\quad \lambda\rightarrow \infty. \label{MakrovPE} 
    \end{equation}
Set
$$
l:=\begin{cases}
2m-1, & {\rm cdeg}_j F\ \mbox{\rm is even},\ j=1,\ldots,p;  \\
2m, & {\rm else}.
\end{cases}
$$ 
Then we call the matrix
    sequence~$\mathscr S:=({\bf s}_k)_{k=0}^{l}$ the sequence of
    Markov parameters 
of~$F(\lambda)$.\label{item:def2_i}

Now we give an algorithm to compute the sequence of Markov parameters. The proposed algorithm relies on the formula 
\begin{equation}\label{AlgEq}
F_{s}(\lambda)=\sum_{k=0}^{\infty} (-1)^{k}\lambda^{-k}{\bf s}_k F_{d}(\lambda),
\end{equation}
which is an immediate combination of \eqref{Ge} and \eqref{MakrovPE}. 
Set that
$t:=\lfloor\tfrac{l}{2}\rfloor
$. Then $l=t+m$. Note that $\deg F_d=m$ and $\deg F_s\leq t$.
Let
$F_d(\lambda)$ and $F_s(\lambda)$ be  of the form 
\begin{align}
&F_d(\lambda)=\sum_{k=0}^{m} A_{k}\cdot {\rm diag}(\lambda^{{\rm cdeg}_1\,F_d-k},\ldots,\lambda^{{\rm cdeg}_p\,F_d-k}), \label{A} \\
&F_s(\lambda)=\sum_{k=0}^{t}B
_{k}\cdot {\rm diag}(\lambda^{{\rm cdeg}_1F_d-k},\ldots,\lambda^{{\rm cdeg}_p\,F_d-k}).\label{B}
\end{align}
Here $A_0, \ldots, A_m$ and $B_0, \ldots, B_t$ are called the  column degree coefficient matrices  of $F_d(\lambda)$ and $F_s(\lambda)$, respectively ($A_0$ is nonsingular upper triangular), which can be easily obtained from the coefficients of $F(\lambda)$.  Comparing the matrix coefficients of both sides in \eqref{AlgEq} yields that 
\begin{equation}\label{S-A}
\begin{bmatrix}
{\bf s}_{t-m+1} & \cdots & {\bf s}_{t}\\
\vdots & \iddots & \vdots\\
{\bf s}_{t} & \cdots & {\bf s}_{l-1}
\end{bmatrix} \begin{bmatrix}
(-1)^{m}A_m\\
\vdots\\
-A_1
\end{bmatrix}=-\begin{bmatrix}
{\bf s}_{t+1}\\
\vdots\\
{\bf s}_{l}
\end{bmatrix}A_0,\quad 
\begin{bmatrix}
(-1)^t B_{t}\\
\vdots\\ 
B_{0}
\end{bmatrix}=
\begin{bmatrix}
{\bf s}_0 & \cdots & {\bf s}_{t}\\
 & \ddots & \vdots\\
& & {\bf s}_0
\end{bmatrix}\begin{bmatrix}
(-1)^t A_t\\
\vdots\\
A_0
\end{bmatrix}.
\end{equation}
The equation \eqref{S-A} reveals a correspondence between
the coefficients and Markov parameters of $F(\lambda)$.
It helps us to derive a recurrence relation for Markov parameters, which is Algorithm \ref{alg1}.

\begin{algorithm}
\caption{}
\begin{multicols}{2}
\begin{algorithmic}[1]\label{alg1}
\REQUIRE The column degree coefficient matrices $A_0, \ldots, A_m, B_0, \ldots, B_t$ of $F_d(\lambda)$ and $F_s(\lambda)$, respectively, given as in \eqref{A}--\eqref{B}
\ENSURE The sequence of Markov parameters of $F(\lambda)$
\STATE $A_{-1}=A_0^{-1}$
\STATE $\mathbf{s}_0=B_{0}A_{-1}$
\STATE $A=\begin{bmatrix}
-A_1\\
\vdots\\
(-1)^m A_m
\end{bmatrix}$
\FOR{$k=1:t$}
\STATE 
$S_k=\begin{bmatrix} {\bf s}_{k-1} & \cdots & {\bf s}_{0}\end{bmatrix}$
    \STATE $\mathbf{s}_k=\big((-1)^kB_k-S_kA_{\{1,\ldots,kp\}}\big)A_{-1}$ \label{Line1}
\ENDFOR
\IF{$m\geq 1$}
\FOR{$k=t+1:l$}
    \STATE $\mathbf{s}_k=-\begin{bmatrix}{\bf s}_{k-1} & \cdots & {\bf s}_{k-m}\end{bmatrix}AA_{-1}$ \label{Line2}
\ENDFOR
\ENDIF
\RETURN  $\mathbf{s}_0, \ldots, \mathbf{s}_{l}$ 
\end{algorithmic}
\end{multicols}
\end{algorithm}

\begin{remark}
In general, there may exist $j\in \{1,\ldots,p\}$ such that ${\rm cdeg}_j\,F_d<k\leq m$ for some $k$. In this case, the $j$-columns of $A_k$ and $B_k$ are both zero vectors and, thus, the corresponding elements of $A$ are zero. The greater the disparity in the column degrees of $F_d(\lambda)$, the more zero elements  $A$ has, thereby enhancing the computational efficiency of all matrix multiplication in Lines \ref{Line1} and \ref{Line2} of Algorithm \ref{alg1}. 
\end{remark}

For simplicity, we write $\mathcal I_i$ for $\mathcal I_i^c(F_d)$ for $i=-1,0,\ldots,m$.
 Let us continue our procedure to associate two structured matrices with rectangular Hankel blocks:

\begin{enumerate}[label=\textbf{Step \arabic*}, start=4, align=left, leftmargin=*]
\item \label{step4} For the dominant part $F_d(\lambda)$ of $F(\lambda)$, we divide its column index sets $\mathcal I_i$ ($i=-1,\ldots,\deg F_d$) by its even/odd column degrees as
\begin{align*}
&\mathcal I_i^e:=\{k : i< {\rm cdeg}_kF_d\leq \deg F_d,\ {\rm cdeg}_kF  {\rm\ is\ even}\},\\
&\mathcal I_i^o:=\{k : i< {\rm cdeg}_kF_d\leq \deg F_d,\ {\rm cdeg}_kF  {\rm\ is\ odd}\}.
\end{align*}
\end{enumerate}
 Define the shifted column index sets 
$$
\wtilde {\mathcal I}_i:=\mathcal I_i^e\cup \mathcal I_{i-1}^o,\quad  i=0,\ldots,\deg F_d.
$$
Obviously $\mathcal I_i=\mathcal I_i^e\cup \mathcal I_i^o$ and thus $\mathcal I_i \subseteq \wtilde {\mathcal I}_i$ for $i=0,\ldots,\deg F_d$.
 From each Markov parameter ${\bf s}_k$, we extract some submatrices by retaining the rows and columns indexed by $\mathcal I_i$ or $\wtilde {\mathcal I}_i$  as 
$$\wtilde {\bf h}_{i,j}:=\left({\bf s}_{i+j}\right)_{\wtilde {\mathcal I}_i}^{\wtilde {\mathcal I}_j}\in \mathbb{C}^{|\wtilde {\mathcal I}_i|\times |\wtilde {\mathcal I}_j|},\quad {\bf  h}_{i,j}:=\left({\bf s}_{i+j+1}\right)_{\mathcal I_i}^{\mathcal I_j}\in \mathbb{C}^{|\mathcal I_i|\times |\mathcal I_j|}.
$$

\begin{remark}\label{RemtildeFd}
The class $\{\wtilde {\mathcal I}_i\}_{i=0}^{\deg F_d}$ constitutes a nested family of subsets:
    $$ \left\{ k : \operatorname{cdeg}_k F_d = \deg F_d,\ \operatorname{cdeg} F_d {\rm\ is\ odd} \right\}=\wtilde {\mathcal I}_{\deg F_d}\subseteq \wtilde {\mathcal I}_{\deg F_d-1} \subseteq \cdots \subseteq \wtilde {\mathcal I}_0=\{1,\ldots,p\}.$$
\end{remark}

\begin{remark}\label{RemtI}
If $F(\lambda)$ is monic,     $
\wtilde {\mathcal I}_i=\{1,\ldots,p\}$ for $i=0,\ldots,\deg F_d-1$. Further, Remark \ref{RemtildeFd} shows that $\wtilde {\mathcal I}_{\deg F_d}=\{1,\ldots,p\}$ when $F(\lambda)$ is monic of odd degree.
    
\end{remark}

\begin{enumerate}[label=\textbf{Step \arabic*}, start=5, align=left]

\item is divided into two cases: \label{step5}
\begin{itemize}
\item If all column degrees of $F(\lambda)$ are even, we construct the rectangular block Hankel matrix $\mathbf{H}_{0}^{\mathscr S}$ and the shifted rectangular block Hankel matrix $\mathbf{H}_{1}^{\mathscr S}$ associated with $\mathscr S$
\begin{equation}\label{HankelMatrix}
\mathbf{H}_{0}^{\mathscr S}:=(\wtilde {\bf  h}_{i,j})_{i,j=0}^{m-1},\quad \mathbf{H}_{1}^{\mathscr S}:=({\bf h}_{i,j})_{i,j=0}^{m-1}.
\end{equation}

\item  In cases where 
$F(\lambda)$ contains at least one column of odd degree, two corresponding rectangular block Hankel matrices are set to be
\begin{equation}\label{HankelMatrixII}
\mathbf{H}_{0}^{\mathscr S}:=(\wtilde {\bf  h}_{i,j})_{i,j=0}^{m},\quad \mathbf{H}_{1}^{\mathscr S}:=({\bf h}_{i,j})_{i,j=0}^{m-1}.
\end{equation}
\end{itemize}
\end{enumerate}

A sample for the construction of $\mathbf{H}_{0}^{\mathscr S}$ by \ref{step4}--\ref{step5} is illustrated in  Fig. \ref{fig:main1}.
\begin{figure}[H]\centering
\includegraphics[width=\textwidth]{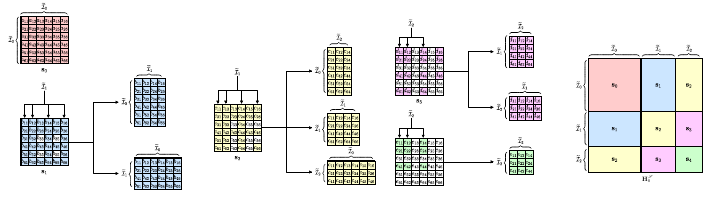}
    \caption{The construction of $\mathbf{H}_{0}^{\mathscr S}$ from a sequence of $6\times 6$ Markov parameters $({\bf s}_k)_{k=0}^4$}
    \label{fig:main1}
\end{figure}

\begin{remark}
In the cases when $F(\lambda)$ is a scalar real polynomial or monic matrix polynomial, the even- and odd-degree cases differ in both the Markov parameters and the associated Hankel matrices (see \cite[Theorem~17, Ch.~XV]{Gan} and \cite[Definition~2.2 and Remark~2.3]{XD}).
For the nonmonic case, what matters is no longer the parity of
the degree of $F(\lambda)$ but the parities of the individual column
degrees. The dichotomy ``all column degrees even / at least one column of
odd degree'' takes over the role of the classical even/odd degree dichotomy.
If all column degrees are even, the length of the Markov parameter sequence
is $2m$, and both $\mathbf H_0^{\mathscr S}$ and $\mathbf H_1^{\mathscr S}$ consist of $m \times m$ blocks. If at least one column degree is odd,
the length increases to $2m+1$ and $\mathbf H_0^{\mathscr S}$ is enlarged
to an $(m+1) \times (m+1)$ block matrix. Moreover, the (shifted) column index sets for the submatrix extraction
of the Markov parameters in \ref{step4} are determined by the column degrees of
$F(\lambda)$ (and of its dominant part $F_d(\lambda)$).
\end{remark}

\subsection{Hurwitz stability criterion}\label{subsec4.2}

Let $F(\lambda)\in\mathbb C[\lambda]^{p\times p}$ be column reduced of nonzero degree $2m$ or $2m+1$. For testing the Hurwitz stability of $F(\lambda)$, we associate with the sequence of Markov parameters $\mathscr S:=({\bf s}_{k})_{k=0}^{l}$ and rectangular block Hankel matrices via adaptive ‌column-wise Li\'{e}nard-Chipart splitting (see \ref{step1}--\ref{step5}).  It deserves special attention if the sequence of Markov parameters
    $\mathscr S$  is Hermitian, that is,  each Markov parameter ${\bf s}_k$ is Hermitian for $k= 0,\ldots,l$.
In what follows, we demonstrate that the Hermitian property of   $\mathscr S$ induces some alternative  Hermitian or symmetric properties:
\begin{proposition}\label{ProSelfAdjoint}
Let $F(\lambda)\in\mathbb C[\lambda]^{p\times p}$ be column reduced with the sequence of Markov parameters
$\mathscr S:=({\bf s}_{k})_{k=0}^{l}$. Then $\mathscr S$ is Hermitian if and only if $R_F(\lambda)$ is symmetric with respect to $\mathbb R$. In this case,
$\mathbf{H}_{0}^{\mathscr S}$ and $\mathbf{H}_{1}^{\mathscr S}$ are Hermitian matrices.
\end{proposition}
\begin{proof}
The ``if'' implication is obvious. 

The proof for the ``only if'' implication: 
Let $$C=\begin{bmatrix}0 &\cdots&0& (-1)^{m-1} A_mA_0^{-1}\\ I_p &&& (-1)^mA_{m-2}A_0^{-1}\\&\ddots&&\vdots\\&&I_p &A_1A_0^{-1}\end{bmatrix},\quad \mathscr{H}_k=\begin{bmatrix}
{\bf s}_{k} & \cdots & {\bf s}_{k+m-1}\\
\vdots & \iddots & \vdots\\
{\bf s}_{k+m-1} & \cdots & {\bf s}_{k+2m-2}
\end{bmatrix}.$$ In view of \eqref{AlgEq}, if $\mathscr{H}_k$ is Hermitian, we have
$$
C^*\mathscr{H}_k C=\mathscr{H}_{k+2},
$$
which implies $\mathbf{s}_{k+2m-1}$ and $\mathbf{s}_{k+2m}$ are Hermitian. The Hermitian property of $\mathbf{s}_0,\ldots\mathbf{s}_{l}$ implies that $\mathscr{H}_{0}$ is Hermitian.  
By induction, it follows that $\mathbf{s}_k$ is Hermitian for all $k\geq 0$, which yields that $R_F(\lambda)$ is symmetric with respect to $\mathbb R$ (see, e.g., \cite{Rudin}). Further, if $\mathscr{S}$ is Hermitian, it is clear from the construction of $\mathbf{H}_0^{\mathscr S}$ and $\mathbf{H}_1^{\mathscr S}$ that $\mathbf{H}_0^{\mathscr S}$ and $\mathbf{H}_1^{\mathscr S}$ are Hermitian matrices. 
\end{proof}

The following result establishes a structural bridge between the inertia of regular matrix polynomials and the finite B\'ezoutian matrix.

\begin{lemma}\label{lambdaoutInertia} 
Let $L(\lambda),L_1(\lambda)\in \mathbb C[\lambda]^{p\times p}$ be column reduced such that
\begin{equation*}
L^{\vee}_1(\lambda)L_1(\lambda)=L^{\vee}(\lambda)L(\lambda). 
\end{equation*}
    Then
\begin{align}
\gamma_+(L)&=\pi(-{\rm i} {\bf B}(L_1^{\vee},L_1;L^{\vee},L)) + \gamma_+(L_0), \nonumber  \\
\gamma_-(L)&=\nu(-{\rm i} {\bf B}(L_1^{\vee},L_1;L^{\vee},L)) + \gamma_-(L_0), \nonumber   \\
\gamma_{0}(L)&=\delta(-{\rm i} {\bf B}(L_1^{\vee},L_1;L^{\vee},L))-\gamma_+(L_0)-\gamma_-(L_0), \label{gamma0}
\end{align}
where~$L_0(\lambda)$ is a GCRD of~$L(\lambda)$ and~$L_1(\lambda)$.
\end{lemma}

\begin{proof}  In view of \cite[Theorem 2.1]{LT82}, we have
\begin{align}
\gamma_+(L)&=\pi(-{\rm i} {\bf B}_{n,n}(L_1^{\vee},L_1;L^{\vee},L)) + \gamma_+(L_0), \nonumber  \\
\gamma_-(L)&=\nu(-{\rm i} {\bf B}_{n,n}(L_1^{\vee},L_1;L^{\vee},L)) + \gamma_-(L_0),  \nonumber  \\
\gamma_{0}(L)+\gamma_{\infty}(L)&=\delta(-{\rm i} {\bf B}_{n,n}(L_1^{\vee},L_1;L^{\vee},L))-\gamma_+(L_0)-\gamma_-(L_0).  \label{gamma0'}
\end{align}
where $n\coloneqq\max\{\deg L, \deg L_1\}$.
Note that (see, e.g., \cite{GHN})
\begin{equation}\label{gammainfty}
\gamma_{\infty}(L)=np-\deg \det L.
\end{equation}
Then a combination of \eqref{gamma0'} and \eqref{gammainfty} yields \eqref{gamma0}.

\end{proof}

\begin{lemma}\label{LemIndexDFd} Let $F(\lambda)\in\mathbb C[\lambda]^{p\times p}$ be column reduced with the dominant part $F_d(\lambda)$. Further, let
\begin{equation}\label{wtildeD}
\wtilde D(\lambda):= F_d(-\lambda^2)\alpha({\rm i}\lambda),
\end{equation}
where $\alpha(\lambda)$ is given as in \eqref{alphaz}.
Then, for $j=0,\ldots,\deg F_d,$
$$
\mathcal I_{2j+1}^c(\wtilde D)=\mathcal I_{j},\quad \mathcal I_{2j}^c(\wtilde D)=\wtilde {\mathcal I}_j.
$$
\end{lemma}

\begin{proof}
Note that 
\begin{equation}\label{DFd}
{\rm cdeg}_k\wtilde D=2{\rm cdeg}_kF_d+b_k,
\end{equation}
where 
$$
b_k:=\begin{cases}
0, & {\rm cdeg}_kF  {\rm\ is\ even},\\
1, & {\rm cdeg}_kF  {\rm\ is\ odd}.
\end{cases}
$$
It follows that
$$
\mathcal I_{2j+1}^c(\wtilde D)=\{k : j<{\rm cdeg}_kF_d+\frac{b_k-1}{2}\}=\mathcal I_{j}.
$$

Due to \eqref{DFd} again we obtain that 
\begin{align*}
\mathcal I^e_{j}&=\{k  : 2j<{\rm cdeg}_k\wtilde D,b_k=0\},\\
 \mathcal I_{j-1}^o&=\{k  : j-1<{\rm cdeg}_kF_d,b_k=1\}
 =\{k : 2j<{\rm cdeg}_k\wtilde D,b_k=1\}.
\end{align*}
Subsequently, 
$$
\mathcal I_{2j}^c(\wtilde D)=\mathcal I_j^e\cup \mathcal I_{j-1}^o=\wtilde {\mathcal I}_j.
$$
\end{proof}

In what follows, we assume without loss of generality that $F_{\mathrm{hcdc}}$ is nonsingular upper triangular, in view of the preprocessing step in Section \ref{subsec4.1}.

\begin{theorem}\label{ThmInertia}
   Let $F(\lambda)\in\mathbb C[\lambda]^{p\times p}$ of nonzero degree with the nonsingular upper triangular $F_{\mathrm{hcdc}}$ and the Hermitian sequence of Markov parameters $\mathscr S$. 
Further let $
\mathbf{H}_{0}^{\mathscr S}$ and $\mathbf{H}_{1}^{\mathscr S}$ be defined as in \eqref{HankelMatrix}--\eqref{HankelMatrixII}. Then
      \begin{align*}
        \gamma'_-(F)&=\pi(\mathbf{H}_{0}^{\mathscr S})+\pi(\mathbf{H}_{1}^{\mathscr S}) +\gamma'_-(G),\\
        \gamma'_+(F)&=\nu(\mathbf{H}_{0}^{\mathscr S})+\nu(\mathbf{H}_{1}^{\mathscr S}) + \gamma'_+(G),\\
        \gamma'_0(F)&=\delta(\mathbf{H}_{0}^{\mathscr S})+\delta(\mathbf{H}_1^{\mathscr S})-\gamma'_+(G)- \gamma'_-(G),
      \end{align*}
where $G(\lambda)$ is a GCRD of~$D(\lambda)$ and~$S(\lambda)$ given as in \eqref{DS}.
\end{theorem}

\begin{proof}
We provide a proof only for the case where all column degrees of $F(\lambda)$ are even; the other case follows by analogous arguments.

Let $F(\lambda)$ be of nonzero degree $2m$ or $2m+1$ and let 
$L(\lambda):=F({\rm i}\lambda)$.
Clearly,
$
L(\lambda)=\wtilde D(\lambda)-{\rm i}\wtilde S(\lambda)
$
where $\wtilde D(\lambda)$ is given as in \eqref{wtildeD} and 
$
\wtilde S(\lambda):={\rm i}S({\rm i}\lambda)= \lambda^{-1} F_s(-\lambda^2)\alpha({\rm i}\lambda).
$ Further let
$
L_1(\lambda):=\wtilde D(\lambda)+{\rm i} \wtilde S(\lambda).
$
Due to Proposition \ref{ProSelfAdjoint}, $R_F(\lambda)$ is symmetric with respect to $\mathbb R$. Then ${\bf B}(L_1^{\vee},L_1;L^{\vee},L)$ and ${\bf B}(\wtilde D^{\vee},\wtilde S;\wtilde S^{\vee},\wtilde D)$ are both well-defined and, noting that $\mathcal I_c(L)=\mathcal I_c(\wtilde D)$, their relation is clear that
 \begin{equation}\label{-BB}
 -{\rm i}{\bf B}(L_1^{\vee},L_1;L^{\vee},L)=2{\bf B}(\wtilde D^{\vee},\wtilde S;\wtilde S^{\vee},\wtilde D).
 \end{equation}

Let $\mathscr S=({\bf s}_k)_{k=0}^{l}$. Note that $\wtilde S(\lambda)(\wtilde D(\lambda))^{-1}$ is strictly proper. 
In view of Remark \ref{RemIndex} and Proposition \ref{LemHs}, we obtain  the structure of $\mathbf{H}(\wtilde D^{\vee},\wtilde S^{\vee};\wtilde S,\wtilde D)$ as
\begin{equation*}
  \begin{bmatrix} {\bf s}_{0,0} & 0 & {\bf s}_{0,2} & \cdots & {\bf s}_{0,2m-2} & 0 \\
        0 &  {\bf s}_{1,1} & 0 & \cdots &0 & {\bf s}_{1,2m-1}  \\
        {\bf s}_{2,0} & 0 & {\bf s}_{2,2} & \cdots & {\bf s}_{2,2m-2} & 0\\
        \vdots & \vdots & \vdots && \vdots & \vdots\\
        {\bf s}_{2m-2,0} & 0 & {\bf s}_{2m-2,2} & \cdots & {\bf s}_{2m-2,2m-2} & 0\\
        0 &  {\bf s}_{2m-1,1} & 0 & \cdots & 0 &  {\bf s}_{2m-1,2m-1}
\end{bmatrix},
\end{equation*}
where ${\bf s}_{i,j}:=\big({\bf s}_{\frac{i+j}{2}}\big)_{\mathcal I^c_i(\wtilde D)}^{\mathcal I^c_j(\wtilde D)}\in |\mathcal I^c_i(\wtilde D)|\times |\mathcal I^c_j(\wtilde D)|$.
From Lemma \ref{LemIndexDFd} one can rewrite $\mathbf{H}(\wtilde D^{\vee},\wtilde S^{\vee};\wtilde S,\wtilde D)$ into the block form
\begin{equation*}
  \begin{bmatrix} \wtilde {\bf h}_{0,0} & 0 & \wtilde {\bf h}_{0,1} & \cdots & \wtilde {\bf h}_{0,m-1} & 0 \\
       0  &  {\bf h}_{0,0} & 0 & \cdots & 0 & {\bf h}_{0,m-1} \\
        \wtilde {\bf h}_{1,0} & 0 & \wtilde {\bf h}_{1,1} & \cdots & \wtilde {\bf h}_{1,m-1} & 0\\
        \vdots & \vdots & \vdots && \vdots & \vdots\\
        \wtilde {\bf h}_{m-1,0} & 0& \wtilde {\bf h}_{m-1,1} & \cdots & \wtilde {\bf h}_{m-1,m-1} & 0\\
        0 &  {\bf h}_{m-1,0} & 0 & \cdots & 0 & {\bf h}_{m-1,m-1}
\end{bmatrix}.
\end{equation*}
Hence, a combination of Proposition \ref{LemBH} and \eqref{-BB} deduces that  $-{\rm i}{\bf B}(L_1^{\vee},L_1;L^{\vee},L)$ is subsequently congruent to
$
 \begin{bmatrix}
 \mathbf{H}_{0}^{\mathscr S} & 0\\
0 & \mathbf{H}_{1}^{\mathscr S}
 \end{bmatrix}.
$
Since 
$(L_1)_{\mathrm{hcdc}}=\wtilde D_{\mathrm{hcdc}}$ is nonsingular upper triangular,
 $L_1(\lambda)$ is column reduced and hence regular as well. Note that, due to \cite[Propositions~A.3--A.5]{XD}, $G({\rm i}\lambda)$ is a GCRD of~$L(\lambda)$ and~$L_1(\lambda)$. Therefore, the remainder of the proof is an immediate consequence of Lemma \ref{lambdaoutInertia}.

\end{proof}

Analogous to the proof of \cite[Theorem 4.4]{XD}, from Theorem \ref{ThmInertia} we can obtain the following Hurwitz stability criterion.
\begin{theorem}\label{ThmHurwitz}
 Let $F(\lambda)\in\mathbb C[\lambda]^{p\times p}$ of nonzero degree with the nonsingular upper triangular $F_{\mathrm{hcdc}}$ and the Hermitian sequence of Markov parameters $\mathscr S$. Further let $
\mathbf{H}_{0}^{\mathscr S}$ and $\mathbf{H}_{1}^{\mathscr S}$ be defined as in \eqref{HankelMatrix}--\eqref{HankelMatrixII}. Then $F(\lambda)$ is Hurwitz stable if and only if  $\mathbf{H}_{0}^{\mathscr S}$ and $\mathbf{H}_{1}^{\mathscr S}$ are both positive definite.
\end{theorem}

\begin{proof}
    The proof for ``if'' implication:  Assume that $\mathbf{H}_{0}^{\mathscr S}$ and $\mathbf{H}_{1}^{\mathscr S}$ are both positive definite. Then from the proof of Theorem \ref{ThmInertia}, the sum of 
    $
    \pi(\mathbf{H}_{0}^{\mathscr S})
    $ and $
    \pi(\mathbf{H}_{1}^{\mathscr S})
    $ equals the size of $\mathbf{H}(\wtilde D^{\vee},\wtilde S^{\vee};\wtilde S,\wtilde D)$, that is $\deg\det \wtilde D(\lambda)$.
      Accordingly, Theorem \ref{ThmInertia} yields that
\begin{equation*}
\deg\det F(\lambda)\geq \gamma'_-(F)\geq \pi(\mathbf{H}_{0}^{\mathscr S})+ \pi(\mathbf{H}_{1}^{\mathscr S})=\deg\det \wtilde D(\lambda)=\deg\det F(\lambda).
\end{equation*}
 Thus $\gamma'_-(F)=\deg\det F(\lambda)$, which means that $F(\lambda)$ is Hurwitz stable.

 The proof for ``only if'' implication: Let $F(\lambda)$ be Hurwitz stable and let $D(\lambda)$ and~$S(\lambda)$ be given as in \eqref{DS}. Assume that $G(\lambda)$ is a GCRD of~$D(\lambda)$ and~$S(\lambda)$ of the form
$$
D(\lambda)=D_l(\lambda)G(\lambda),\quad S(\lambda)=S_l(\lambda)G(\lambda).
$$  
Then
\begin{equation}\label{FG}
F(\lambda)=(D_l(\lambda)+S_l(\lambda))G(\lambda),\quad F(-\lambda)=(D_l(\lambda)-S_l(\lambda))G(\lambda)\alpha(\lambda)^{-1}\alpha(-\lambda).
\end{equation}
Suppose that there exists a nonzero eigenvalue $\lambda_0$ of $G(\lambda)$. It follows from \eqref{FG} that $\pm \lambda_0\in \sigma(F)$, which contradicts the Hurwitz stability of $F(\lambda)$. Consequently, $\gamma'_-(G)=0$ and, due to Theorem \ref{ThmInertia},
$$
\deg\det F(\lambda)=\pi(\mathbf{H}_{0}^{\mathscr S})+ \pi(\mathbf{H}_{1}^{\mathscr S}),
$$
 which ensures the positive definiteness of $\mathbf{H}_{0}^{\mathscr S}$ and $\mathbf{H}_{1}^{\mathscr S}$.

\end{proof}

\begin{remark}
\label{RemReduction}
For the Hurwitz stability criteria via Markov parameters, the transition from the scalar case \cite[Theorem~17, Ch.~XV]{Gan} to monic matrix polynomials \cite[Theorems~4.4 and 4.10]{XD} is particularly evident in the odd-degree case (see \cite{XD}).
The differences between nonmonic regular matrix polynomials and monic matrix polynomials are even more substantial. Indeed, for a regular matrix polynomials $F(\lambda)\in\mathbb{C}[\lambda]^{p\times p}$, observe that the standard Liénard–Chipart splitting of $F(\lambda)$ may yield singular even and odd parts, $F_e(\lambda)$ and $F_o(\lambda)$, respectively, thereby precluding the direct construction of the desired Markov parameters.
To overcome this difficulty, we first apply a preprocessing transformation to $F(\lambda)$ to obtain a column-reduced form with $F_{\rm hcdc}$ nonsingular and upper triangular. The resulting framework further involves five nontrivial steps: an adaptive ‌column-wise Li\'{e}nard-Chipart splitting in \ref{step1}--\ref{step3}, applied to each column rather than to the matrix polynomial as a whole, to construct the desired Markov parameters, and the assembly in \ref{step4}--\ref{step5} of two structured matrices, $\mathbf{H}_1^{\mathscr S}$ and $\mathbf{H}_0^{\mathscr S}$, which no longer have the standard Hankel or block-Hankel form encountered in the scalar and monic matrix-polynomial cases, but still retain a rectangular block-Hankel structure built from the compressed Markov parameters. Finally, Theorem \ref{ThmHurwitz} assesses the Hurwitz stability of $F(\lambda)$ in the nonmonic case in terms of the positive definiteness of $\mathbf{H}_1^{\mathscr S}$ and $\mathbf{H}_0^{\mathscr S}$. 
\end{remark}

\begin{remark}
\label{RemReduction1}
Specializing Theorem \ref{ThmHurwitz} to monic matrix polynomials recovers the Markov-parameter-based  Hurwitz stability criterion \cite[Theorem 4.4]{XD}. More specifically, under the same assumption of Theorem \ref{ThmInertia}, we particularly set $F(\lambda)$ to be monic of nonzero degree $n=2m$ or $n=2m+1$.
Then 
$$
F_d(\lambda)=\begin{cases}
F_e(\lambda), & \mbox{if } n\mbox{ is even};\\
F_o(\lambda), & \mbox{if } n \mbox{ is odd},
 \end{cases}\quad
 F_s(\lambda)=\begin{cases}
\lambda F_o(\lambda), & \mbox{if }n\mbox{ is even};\\
F_e(\lambda), & \mbox{if }n\mbox{ is odd},
 \end{cases}
$$
where $F_e(\lambda)$ and $F_o(\lambda)$ are, respectively, the even part and odd part of $F(\lambda)$ (see \cite[Definition 2.1]{XD}) satisfying~that
$$
F(\lambda)=F_e(\lambda^2)+\lambda F_o(\lambda^2).
$$
In this case, $\mathscr S:=({\bf s}_k)_{k=0}^{n-1}$ coincides with the truncated sequence of Markov parameters (resp. of the second type) of $F(\lambda)$
defined in \cite[Definition 2.2]{XD} for $n=2m$ (resp. $n=2m+1$). 
Further, 
\begin{equation*}
\mathcal I_{i}=\wtilde {\mathcal I}_i=\{1,\ldots,p\},  \quad i=0,\ldots,m-1,
\end{equation*}
and, by Remark \ref{RemtildeFd}, $\wtilde {\mathcal I}_{m}=\{1,\ldots,p\}$ for $n=2m+1$. It follows that
$$
\mathbf{H}_{0}^{\mathscr S}=\begin{cases}
({\bf s}_{i+j})_{i,j=0}^{m-1}, & n=2m;\\
({\bf s}_{i+j})_{i,j=0}^{m}, & n=2m+1,
\end{cases}\quad \mathbf{H}_{1}^{\mathscr S}=({\bf s}_{i+j+1})_{i,j=0}^{m-1}.
$$
Therefore, Theorems \ref{ThmInertia} and 
\ref{ThmHurwitz} reduce to the inertia representation  \cite[Theorem 3.5]{XD} and the Hurwitz stability criterion \cite[Theorem 4.4]{XD} for the monic case, respectively. 
\end{remark}

\begin{remark}\label{RemFeedback} Although the Hermitian property of $\mathscr S$ in Theorem \ref{ThmHurwitz} may not be satisfied for all regular matrix polynomials, it may be ensured through appropriate feedback control design in practice.
    For example, consider a linear damped mechanical system 
    \begin{equation}\label{MDK}
    M y^{(2)}(t)+D\dot{y}(t)+K y(t)=Bu(t),
    \end{equation}
where $M$ and $B$ are both nonsingular. For stabilization, a feedback $u(t)=F_1\dot{y}(t)+F_0 y(t)$ 
can be designed with the gain matrices chosen to be
$$F_1=B^{-1}(D- {\bf s}_0 M),
\quad F_0=B^{-1}(K- {\bf s}_0^{-1} {\bf s}_1 M)
$$
for any desired positive definite matrices ${\bf s}_0$ and ${\bf s}_1$.
In fact, for the corresponding characteristic matrix polynomial of the
closed-loop system for \eqref{MDK}, the sequence of Markov parameters
$\mathscr S=({\bf s}_k)_{k=0}^1$ is Hermitian, with ${\bf s}_0$ and
${\bf s}_1$ positive definite. By Theorem \ref{ThmHurwitz}, the closed-loop
system is asymptotically stable. We note that such a design is available
whenever the input channel permits the assignment of the closed-loop
coefficient matrices, as is the case for \eqref{MDK} by the formulas above.

\end{remark}

\begin{remark}
As required by the Hurwitz stability condition in Theorem~\ref{ThmHurwitz}, 
the positive definiteness of two Hermitian
block matrices
$\mathbf{H}_{0}^{\mathscr{S}}\in\mathbb{C}^{n_{0}\times n_{0}}$ and
$\mathbf{H}_{1}^{\mathscr{S}}\in\mathbb{C}^{n_{1}\times n_{1}}$
 can be verified numerically by the
Cholesky decomposition, at a
cost of approximately $n_{0}^3/3+n_{1}^3/3$ flops  (real arithmetic count) (see, e.g., \cite{GV}).
Moreover,
the memory requirement of the proposed criterion is also dominated by $\mathbf{H}_{0}^{\mathscr{S}}$ and
$\mathbf{H}_{1}^{\mathscr{S}}$, which require $n_{0}^{2}+n_{1}^{2}$ complex scalars. The Markov parameters used during their assembly form a temporary workspace of no larger size and are released afterwards.
Observe that $n_{0}+n_{1}=\deg\det F(\lambda)\le np$, where the inequality
is strict when the leading coefficient of $F(\lambda)$ is singular, whereas the companion pencil formed by \texttt{polyeig} consists of two $np\times np$ matrices and thus occupies $2n^{2}p^{2}$ complex scalars, excluding the same-order workspace of the subsequent QZ computation. Even in the least
favorable case when $F(\lambda)$ is monic, where no compression of $\mathbf{H}_{0}^{\mathscr{S}}$ and
$\mathbf{H}_{1}^{\mathscr{S}}$ arises such that
$n_{0}=\lceil n/2\rceil p$ and $n_{1}=\lfloor n/2\rfloor p$, the
storage amounts to only one quarter of that of the companion pencil, plus the lower-order term $p^{2}/2$ when $n$ is odd.
\end{remark}

In what follows, we provide numerical examples to illustrate Theorem \ref{ThmHurwitz}. 
All numerical tests were performed in MATLAB R2024b on a Windows 11 machine with an Intel Core i7-14700KF processor (3.40\,GHz, 20 cores) and 32\,GB of memory.
\begin{example}
Consider a matrix polynomial $F(\lambda)\in \mathbb C[\lambda]^{4\times 4}$ of degree $3$ 
\[\resizebox{\linewidth}{!}{$\begin{bmatrix} \lambda^{3} +  \left(1 - {\rm i}\right)\lambda^{2}  +  \left(-4 + 9 {\rm i}\right)\lambda  - 57 + 43 { \rm i}&  \left(2 + {\rm i}\right)\lambda  + 2 - 3 { \rm i}&  \left(-1 + 3 {\rm i}\right)\lambda  - 6 - 4 { \rm i}& - 2  \lambda - 2 - 2 {\rm i}\\- 2 { \rm i} \lambda^{2} +  \left(12 + 17 {\rm i}\right)\lambda  - 39 + 76 { \rm i}&  \lambda^{2} + 4  \lambda - 2 + 4 { \rm i}& {\rm i} \lambda^{2} + 5 {\rm i} \lambda - 2 &  \left(1 - {\rm i}\right)\lambda  - 1 - 3 {\rm i}\\ \lambda^{3} +  \left(1 + {\rm i}\right)\lambda^{2}  +  \left(36 - {\rm i}\right)\lambda  + 29 - 51 {\rm i}&  \left(-2 + {\rm i}\right)\lambda  - 6 + 7 {\rm i}&  \left(-1 - 3 {\rm i}\right)\lambda  - 12 + 4 {\rm i}& 2 + 2 {\rm i}\\ \left(-5 + 2 {\rm i}\right)\lambda  + 12 + 17 {\rm i}& {\rm i} \lambda^{2} + 10 + 16 {\rm i}&  \lambda^{2} +  \lambda + 14 + 8 {\rm i}&  \left(-1 + {\rm i}\right)\lambda  - 1 + {\rm i}\end{bmatrix}.$}\]
Since the leading coefficient of $F(\lambda)$ is singular and non-Hermitian, some recent Hurwitz stability criteria (e.g., \cite[Theorem 6.8]{MMW}, \cite[Theorems 3.5 and 4.4]{XD}) are not applicable. Since $F_{\mathrm{hcdc}}$ is nonsingular, $F(\lambda)$ is column reduced. Applying \ref{step1}--\ref{step2} after a QR decomposition of $F_{\mathrm{hcdc}}$ yields a pair of matrix polynomials
\begin{align*}
&F_{d}(\lambda)=\sqrt2\begin{bmatrix}
-\lambda-16-4{\rm i} & 2-2{\rm i} & 9 & 1\\[1mm]
-7-11{\rm i} & -\lambda-7+3{\rm i} & -3+7{\rm i} & -1\\[1mm]
-6+5{\rm i} & -7-9{\rm i} & -\lambda-7-5{\rm i} & 1\\[1mm]
-20+5{\rm i} & 4-5{\rm i} & 3-4{\rm i} & -1
\end{bmatrix},\\[2mm]
&F_{s}(\lambda)=\sqrt2\begin{bmatrix}
-\lambda+14+4{\rm i} & -{\rm i}\lambda & \lambda & 0\\[1mm]
{\rm i}\lambda+11-32{\rm i} & -2\lambda & -2{\rm i}\lambda & {\rm i}\\[1mm]
\lambda-44-28{\rm i} & 2{\rm i}\lambda & -3\lambda & 2-{\rm i}\\[1mm]
-{\rm i}\lambda-43+47{\rm i} & 2\lambda & 3{\rm i}\lambda & -2-2{\rm i}
\end{bmatrix}.
\end{align*}
Then the sequence of Markov parameters $\mathscr S:=({\bf s}_k)_{k=0}^2$ of $F(\lambda)$ is computed by Algorithm \ref{alg1} as
\begin{align*}
&{\bf s}_0=\begin{bmatrix}1 & {\rm i}& -1 & - {\rm i}\\- {\rm i}& 2 & 2 {\rm i}& -2\\-1 & - 2 {\rm i}& 3 & 3 {\rm i}\\{\rm i}& -2 & - 3 {\rm i}& 4\end{bmatrix},\quad {\bf s}_1= 
\begin{bmatrix}8 & -1 + 4 {\rm i}& -5 + {\rm i}& 4 - 3 {\rm i}\\-1 - 4 {\rm i}& 6 & 2 + {\rm i}& -5 - 3 {\rm i}\\-5 - {\rm i}& 2 - {\rm i}& 4 & -3\\4 + 3 {\rm i}& -5 + 3 {\rm i}& -3 & 6\end{bmatrix},\\[3mm]
& {\bf s}_2=\begin{bmatrix}117 & -56 + 41 {\rm i}& -87 + 26 {\rm i}& 86 - 15 {\rm i}\\-56 - 41 {\rm i}& 92 & 63 & -85 - 41 {\rm i}\\-87 - 26 {\rm i}& 63 & 81 & -69 - 26 {\rm i}\\86 + 15 {\rm i}& -85 + 41 {\rm i}& -69 + 26 {\rm i}& 102\end{bmatrix}.
\end{align*}
\ref{step4} calculates the column index sets and shifted column index sets of $F_d(\lambda)$  
$$\wtilde{\mathcal I}_{0}=\{1,2,3,4\},\quad \wtilde{\mathcal I}_{1}=\{1\},\quad  \mathcal I_{0}=\{1,2,3\},\quad  \mathcal I_{1}=\emptyset.$$
 \ref{step5}  constructs two rectangular block Hankel matrices
$$
\mathbf{H}_{0}^{\mathscr{S}} = 
\begin{bmatrix}1 & {\rm i}& -1 & - {\rm i}& 8\\- {\rm i}& 2 & 2 {\rm i}& -2 & -1 - 4 {\rm i}\\-1 & - 2 {\rm i}& 3 & 3 {\rm i}& -5 - {\rm i}\\{\rm i}& -2 & - 3 {\rm i}& 4 & 4 + 3 {\rm i}\\8 & -1 + 4 {\rm i}& -5 + {\rm i}& 4 - 3 {\rm i}& 117\end{bmatrix},\quad\mathbf{H}_{1}^{\mathscr{S}}=
\begin{bmatrix}8 & -1 + 4 {\rm i}& -5 + {\rm i}\\-1 - 4 {\rm i}& 6 & 2 + {\rm i}\\-5 - {\rm i}& 2 - {\rm i}& 4\end{bmatrix}.
$$
We check that $\mathbf{H}_{0}^{\mathscr{S}}$ and $\mathbf{H}_{1}^{\mathscr{S}}$ are both positive definite.
Therefore, according to Theorem \ref{ThmHurwitz}, $F(\lambda)$ is Hurwitz stable. 

Now we analyze the robustness of the criterion in Theorem \ref{ThmHurwitz} for testing the Hurwitz stability of the aforementioned $F(\lambda)$ under perturbation. 
Suppose that $F(\lambda)$ is written in the form
$$
F(\lambda)=P_3 \lambda^3+P_2 \lambda^2+P_1 \lambda+P_0. 
$$
For preset parameters $\varepsilon$,
we generate $1000$ matrix tuples $(\Delta P_0,\Delta P_1,\Delta P_2,\Delta P_3)$ such that, for each $\Delta P_i$, the $k$-th column contains entries uniformly distributed in the interval $(-1,1)$ if $k\in \mathcal I_{i-1}^c(F)$ and zeros otherwise.
Then, the matrix polynomial $F(\lambda)$ is perturbed to be
$$
\wtilde F(\lambda)=\big(P_3+\varepsilon  \Delta P_3\big) \lambda^3+\big(P_2+\varepsilon \Delta P_2\big) \lambda^2+\big(P_1+\varepsilon \Delta P_1\big) \lambda+ P_0+ \varepsilon \Delta P_0.
$$

\begin{figure}[H]
  \centering
      \begin{minipage}[htbp]{0.48\textwidth}  
        \centering  
        \includegraphics[width=\linewidth]{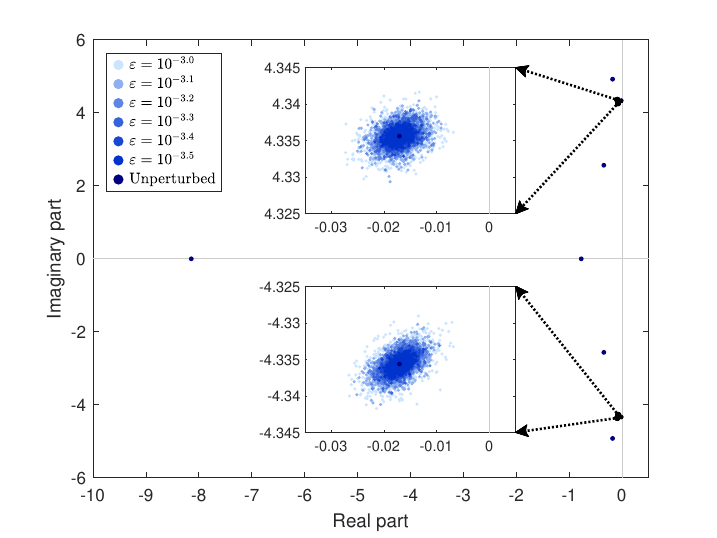}
        \subcaption{The unperturbed and perturbed eigenvalues of $F(\lambda)$}
        \label{fig:1}
    \end{minipage}
    \hfill
    \begin{minipage}[htbp]{0.48\textwidth}
        \centering  
        \includegraphics[width=\linewidth]{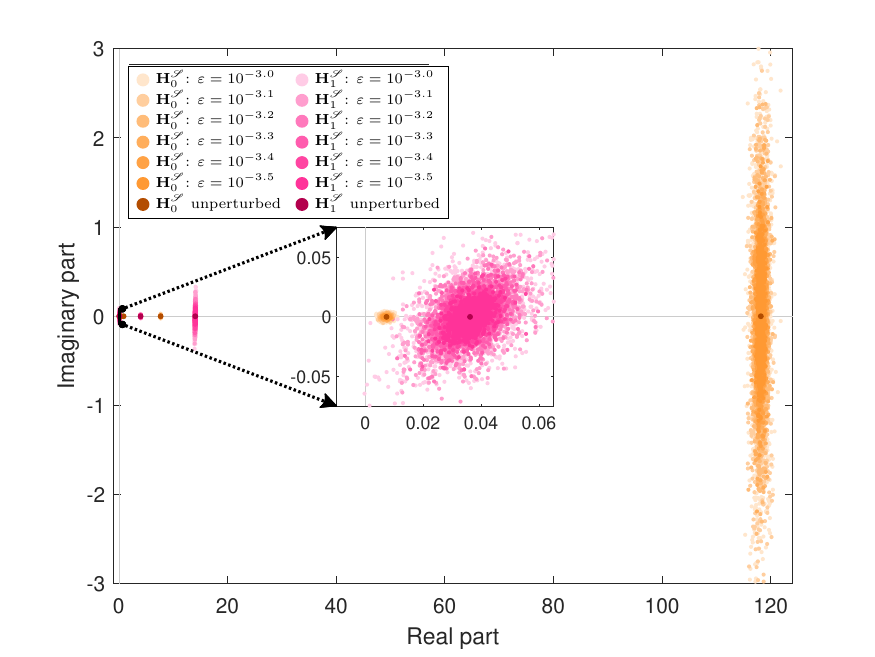}
        \subcaption{The unperturbed and perturbed eigenvalues of $ \mathbf{H}_{0}^{\mathscr S}$ and $\mathbf{H}_{1}^{\mathscr S}$}
        \label{fig:2}
    \end{minipage}
    \begin{minipage}[htbp]{0.48\textwidth}  
        \centering  
        \includegraphics[width=\linewidth]{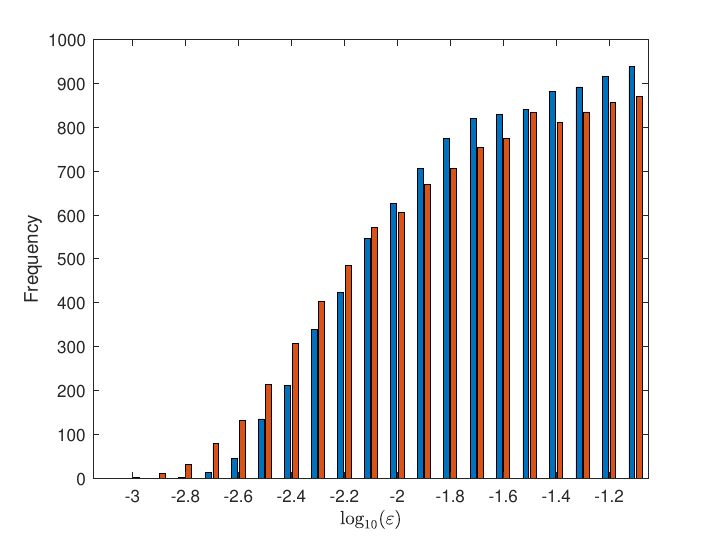}  
        \subcaption{The comparison of the violations of the conditions in \eqref{H01}}
        \label{fig3a}
    \end{minipage}
    \hfill
    \begin{minipage}[htbp]{0.48\textwidth}
        \centering  
        \includegraphics[width=\linewidth]{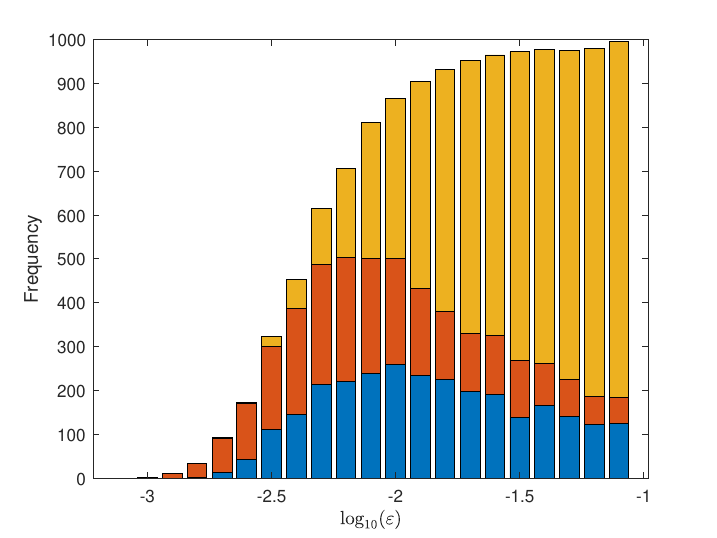}
        \subcaption{The comparison of the failures of the conditions in \eqref{H01}}
        \label{fig3b}
    \end{minipage}
  \caption{} 
  \label{Fig2}
\end{figure}

Fig. \ref{fig:1} superimposes the eigenvalues of $F(\lambda)$  for $\varepsilon=10^{-3-0.1k}$, $k=0,\ldots,5$. The shade darkens as $\varepsilon$ decreases, and the unperturbed eigenvalues are the darkest.
 All unperturbed eigenvalues $\lambda_k(F)$ of $F(\lambda)$ are computed via MATLAB's \texttt{polyeig} function as
\begin{align*}
\lambda_1(F) &= -8.1437,\quad \quad \quad \quad \ \ \  \
\lambda_2(F) = -0.7723,\quad \quad \quad \quad \ \ \  \
\lambda_3(F)= -0.3455 + 2.5642{\rm i}, \\
\lambda_4(F) &= -0.3455 - 2.5642{\rm i},\ \ \ 
\lambda_5(F)  = -0.1795 + 4.9232{\rm i},\ \ \ \,
\lambda_6(F) = -0.1795 - 4.9232{\rm i},\ \ 
\\
\lambda_7(F) &= -0.0170 + 4.3356{\rm i}, \ \ \ 
\lambda_8(F) = -0.0170 - 4.3356{\rm i}. 
\end{align*}
Fig. \ref{fig:2} superimposes the eigenvalues of $\mathbf{H}_{0}^{\mathscr S}$ (orange) and $\mathbf{H}_{1}^{\mathscr S}$ (pink) computed for $\varepsilon=10^{-3-0.1k}$, $k=0,\ldots,5$. The shade darkens as $\varepsilon$ decreases, and the unperturbed eigenvalues are the darkest.
All unperturbed eigenvalues $\lambda_k(\mathbf{H}_0^{\mathscr{S}})$ and $\lambda_k(\mathbf{H}_1^{\mathscr{S}})$ of $\mathbf{H}_0^{\mathscr{S}}$ and $\mathbf{H}_1^{\mathscr{S}}$ respectively are computed via MATLAB's \texttt{eig} function as
\begin{align*}
\lambda_1(\mathbf{H}_0^{\mathscr{S}}) &= 118.1688, \ \ 
\lambda_2(\mathbf{H}_0^{\mathscr{S}}) = 7.6381, \ \ 
\lambda_3(\mathbf{H}_0^{\mathscr{S}}) = 0.8146, \ \ 
\lambda_4(\mathbf{H}_0^{\mathscr{S}}) = 0.3711, \ \
\lambda_5(\mathbf{H}_0^{\mathscr{S}}) = 0.0073, \\ 
\lambda_1(\mathbf{H}_1^{\mathscr{S}}) &= 14.0143, \ \ \ \, 
\lambda_2(\mathbf{H}_1^{\mathscr{S}}) = 3.9496,  \ \  
\lambda_3(\mathbf{H}_1^{\mathscr{S}}) = 0.0361. 
\end{align*}
   The relative perturbation errors $ r_k({F}) $, $ r_k(\mathbf{H}_0^{\mathscr{S}}) $, and $ r_k(\mathbf{H}_1^{\mathscr{S}}) $, corresponding to the eigenvalues $ \lambda_k({F}) $ ($ k = 1, \ldots, 8 $), $ \lambda_k(\mathbf{H}_0^{\mathscr{S}}) $ ($ k = 1, \ldots, 5 $), and $ \lambda_k(\mathbf{H}_1^{\mathscr{S}}) $ ($ k = 1, 2, 3 $) respectively, are summarized in Table~\ref{tab}. 

\begin{table}[t]
\centering
\begin{minipage}{\textwidth}
\centering
\begin{tabular}{l!{\vrule width 1pt}w{c}{1.5cm}w{c}{1.5cm}w{c}{1.5cm}w{c}{1.5cm}w{c}{1.5cm}w{c}{1.5cm}w{c}{1.5cm}w{c}{1.5cm}}
\toprule
$\log_{10}\varepsilon$ & {$r_1({F})$} & {$r_2({F})$} & {$r_3({F})$} & {$r_4({F})$} & {$r_5({F})$} & {$r_6({F})$} & {$r_7({F})$} & {$r_8({F})$} \\
-3.5 & 5.43e-4 & 2.77e-4 & 4.55e-4 & 4.45e-4 & 8.81e-4 & 8.77e-4 & 8.33e-4 & 8.91e-4 \\
-3.4 & 7.89e-4 & 3.56e-4 & 5.58e-4 & 5.92e-4 & 1.16e-3 & 1.26e-3 & 1.08e-3 & 1.14e-3 \\
-3.3 & 9.24e-4 & 4.64e-4 & 6.55e-4 & 7.66e-4 & 1.40e-3 & 1.31e-3 & 1.31e-3 & 1.22e-3 \\
-3.2 & 1.21e-3 & 5.55e-4 & 8.08e-4 & 9.28e-4 & 1.79e-3 & 1.83e-3 & 1.73e-3 & 1.84e-3 \\
-3.1 & 1.59e-3 & 7.75e-4 & 1.19e-3 & 1.06e-3 & 2.24e-3 & 2.34e-3 & 1.91e-3 & 2.10e-3 \\
-3.0 & 2.20e-3 & 9.30e-4 & 1.38e-3 & 1.48e-3 & 2.74e-3 & 2.75e-3 & 2.46e-3 & 2.65e-3 \\
\bottomrule
\end{tabular}
\end{minipage}
\begin{minipage}{\textwidth}
\centering
\begin{tabular}{
  l!{\vrule width 1pt}w{c}{1.5cm}w{c}{1.5cm}w{c}{1.5cm}w{c}{1.5cm}w{c}{1.5cm}!{\vrule width 1pt}w{c}{1.5cm}w{c}{1.5cm}w{c}{1.5cm}
}
\toprule
{$\log_{10}\varepsilon$} & 
{$r_1$($\mathbf{H}_0^{\mathscr{S}}$)} & {$r_2$($\mathbf{H}_0^{\mathscr{S}}$)} & {$r_3$($\mathbf{H}_0^{\mathscr{S}}$)} & {$r_4$($\mathbf{H}_0^{\mathscr{S}}$)} & {$r_5$($\mathbf{H}_0^{\mathscr{S}}$)} & 
{$r_1$($\mathbf{H}_1^{\mathscr{S}}$)} & {$r_2$($\mathbf{H}_1^{\mathscr{S}}$)} & {$r_3$($\mathbf{H}_1^{\mathscr{S}}$)} \\
-3.5 & 1.06e-2 & 1.18e-3 & 1.05e-3 & 1.07e-3 & 2.41e-1 & 6.91e-3 & 1.93e-3 & 7.40e-1 \\
-3.4 & 1.53e-2 & 1.63e-3 & 1.37e-3 & 1.32e-3 & 2.98e-1 & 1.03e-2 & 2.53e-3 & 9.81e-1 \\
-3.3 & 1.95e-2 & 2.14e-3 & 1.79e-3 & 1.71e-3 & 4.33e-1 & 1.42e-2 & 3.27e-3 & 1.56e+0 \\
-3.2 & 2.54e-2 & 2.33e-3 & 2.36e-3 & 2.36e-3 & 4.96e-1 & 1.55e-2 & 4.68e-3 & 1.59e+0 \\
-3.1 & 3.24e-2 & 3.69e-3 & 2.69e-3 & 2.88e-3 & 7.10e-1 & 2.36e-2 & 5.45e-3 & 2.33e+0 \\
-3.0 & 3.75e-2 & 4.25e-3 & 3.26e-3 & 3.15e-3 & 7.50e-1 & 2.75e-2 & 6.87e-3 & 2.61e+0 \\
\bottomrule
\end{tabular}
\end{minipage}
\caption{}
\label{tab}
\end{table}
   
In view of \cite[Claims 2.4--2.5]{XD}, a claim for the robustness of stability criterion in Theorem \ref{ThmHurwitz}  with the Hermitian condition of $\mathscr S$ being relaxed, that is,  
 \begin{equation}\label{H01} 
\wtilde F(\lambda)\ \mbox{is Hurwitz stable if and only if all perturbed eigenvalues of } \mathbf{H}_{0}^{\mathscr S} \mbox{and}\ \mathbf{H}_{1}^{\mathscr S} \mbox{are located in}\ \mathbb C_+ 
\end{equation}
 does not hold in general, where $\mathbb{C}_+$ denotes the open right half-plane. \eqref{H01} may hold for perturbations 
$\varepsilon = 10^{-3.1-0.1k}$ ($k = 0, \dots, 4$), as no violations were observed in Fig. \ref{fig:1} and \ref{fig:2}. Under broader perturbations $\varepsilon=10^{-3.1+0.1k}$ ($k=0,\ldots,20$), Fig.~\ref{fig3a} and \ref{fig3b} quantify the violations and failure instances of the two sides of 
the equivalence \eqref{H01}. 
Fig. \ref{fig3a} compares the violations of the conditions in \eqref{H01}: (i) the number of instability cases of $\wtilde F(\lambda)$ (blue), and (ii) the count of instances where perturbed eigenvalues of  $\mathbf{H}_{0}^{\mathscr S}$ or $\mathbf{H}_{1}^{\mathscr S}$ escape $\mathbb C_+$ (red). Fig. \ref{fig3b} further categorizes the violations into:
 (i) Cases where $\wtilde F(\lambda)$ loses Hurwitz stability despite all eigenvalues of  $\mathbf{H}_{0}^{\mathscr S}$ and
 $\mathbf{H}_{1}^{\mathscr S}$ remain in $\mathbb C_+$ (blue); (ii) Instances where perturbed eigenvalues of $\mathbf{H}_{0}^{\mathscr S}$ or $\mathbf{H}_{1}^{\mathscr S}$ cross or touch ${\rm i}\mathbb R$ while $\wtilde F(\lambda)$ retains Hurwitz stability (red); (iii) Cases concurrently losing Hurwitz stability of $\wtilde F(\lambda)$ and having at least one perturbed eigenvalue of $\mathbf{H}_{0}^{\mathscr S}$ or $\mathbf{H}_{1}^{\mathscr S}$ leaving $\mathbb C_+$ (yellow).
 Our experiments observe a reversal in the relative prevalence around $\varepsilon=10^{-2.1}$: Case (i) versus Case (ii) violations in Fig. \ref{fig3a}, and Case (i) versus Case (ii) failures in Fig. \ref{fig3b}.
 
\end{example}

 \begin{example}\label{DAE3order}
 Consider a third-order differential-algebraic system from \cite[Section V]{DY2006}:
\begin{equation}\label{OpenLoopSystem}
A_3y^{(3)}(t)+A_2\ddot{y}(t)+A_1\dot{y}(t)+A_0y(t)=Bu(t)
\end{equation}
with matrices
$$A_3=\begin{bmatrix}1&0&0\\0&1&0\\0&0&0\end{bmatrix},
A_2=\begin{bmatrix}1&0&0\\0&-1&0\\0&0&1\end{bmatrix},
A_1 = \begin{bmatrix}1&0&0\\1&1&1\\0&1&0\end{bmatrix},
A_0=\begin{bmatrix}0&10&0\\1&0&0\\0&0&-1\end{bmatrix},
B=\begin{bmatrix}1&0&0\\0&0&1\\0&1&0\end{bmatrix}.$$ 
The associated matrix polynomial is
$$F(\lambda)=A_3\lambda^3+A_2\lambda^2+A_1\lambda+A_0.$$
Using MATLAB's \texttt{polyeig} function, we compute the finite eigenvalues of $F(\lambda)$ as
	\[
    1.7639,\quad 0.9313,\quad 0.5667 \pm 1.6369{\rm i}, \quad -0.8932 \pm 0.1167{\rm i}, \quad -1.0211 \pm 1.2074{\rm i}
	\]
	along with one infinite eigenvalue. The result shows that $4$ finite eigenvalues have positive real parts, which implies that $F(\lambda)$ is not Hurwitz stable. Consequently, the system is not asymptotically stable.
We design a feedback in the form
$u(t)=F_2\ddot{y}(t)+F_1\dot{y}(t)+F_0 y(t)$, where one possible choice of the gain matrices, computed from prescribed Markov parameters and kept small in norm, is as 
$$F_2=\begin{bmatrix}-1.6&0&0\\0&0&0\\0&-1.2878&0\end{bmatrix},\quad 
F_1=\begin{bmatrix}0.8&-3.3&0\\0&1&-0.5012\\0.9885&-0.3475&1\end{bmatrix},\quad 
F_0=\begin{bmatrix}-0.519&1.42&0\\0&0&-1.1995\\0.9967&-0.1&0\end{bmatrix}.$$
The closed-loop system becomes
\begin{equation}\label{ClosedLoopSystem}A_3y^{(3)}(t)+(A_2-BF_2)\ddot{y}(t)+(A_1-BF_1)\dot{y}(t)+(A_0-BF_0)y(t)=0,\end{equation}
whose corresponding matrix polynomial is 
$$\tilde{F}(\lambda)=A_3\lambda^3+(A_2-BF_2)\lambda^2+(A_1-BF_1)\lambda+(A_0-BF_0).$$
Via \ref{step1}--\ref{step5}, the associated rectangular block Hankel matrices are computed as
$$\mathbf{H}_0^{\mathscr{S}}=\begin{bmatrix}2.6000&0&0&0.0010&0\\0&0.2878&0&0&0.2878\\0&0&0.5012&0&0\\0.0010&0&0&0.0002&0.0033\\0&0.2878&0&0.0033&0.3878\end{bmatrix},\quad\mathbf{H}_1^{\mathscr{S}}=\begin{bmatrix}0.0010&0&0\\0&0.2878&0\\0&0&0.1000\end{bmatrix}.$$
We verify via Cholesky factorization that $\mathbf{H}_{0}^{\mathscr{S}}$ and $\mathbf{H}_{1}^{\mathscr{S}}$ are both positive definite. Therefore, according to Theorem \ref{ThmHurwitz}, $\tilde{F}(\lambda)$ is Hurwitz stable, which implies the asymptotic stability of the closed-loop system.

Using MATLAB's \texttt{polyeig} function, we compute the finite eigenvalues of $\tilde{F}(\lambda)$ as
	\[
	-0.0230 \pm 0.4079{\rm i},\quad
    -0.0399,\quad
    -0.1010 \pm 1.1623{\rm i},\quad
	-0.2506 \pm 0.3697{\rm i},\quad
	-2.5999,
	\]
	along with one infinite eigenvalue, which is consistent with the Hurwitz stability of $\tilde{F}(\lambda)$. 
We next visualize the dynamics of the open-loop system \eqref{OpenLoopSystem} and the closed-loop system \eqref{ClosedLoopSystem} through numerical simulations. Introducing $z(t)=
\begin{bmatrix}y(t)^{\rm T}&\dot{y}(t)^{\rm T}&\ddot{y}(t)^{\rm T}\end{bmatrix}^{\rm T}$, 
\eqref{OpenLoopSystem} can be rewritten as first-order differential-algebraic system
\begin{equation}\label{FirstOrderDAEExample2}
M\dot{z}(t)=Az(t),
\end{equation}
where $$M=\begin{bmatrix}I_3\\&I_3\\&&A_3\end{bmatrix},\quad A=\begin{bmatrix}&I_3\\&&I_3\\-A_0&-A_1&-A_2\end{bmatrix}.$$
This system \eqref{FirstOrderDAEExample2} is of index-$1$ and solvable via \texttt{ode15s}.
We generate $10$ random tuples $(y(0),\dot{y}(0),\ddot{y}(0))$, where each entry of $y(0)$, $\dot{y}(0)$ and $\ddot{y}(0)$ is independently and uniformly distributed over the interval $(-1,1)$. Since arbitrary choices of these initial values may not satisfy the algebraic constraints of \eqref{FirstOrderDAEExample2}, we employ MATLAB's \texttt{decic} to compute a consistent initial value  $z(0)$ for \eqref{FirstOrderDAEExample2} by adjusting the components of $\ddot{y}(0)$ while retaining $y(0)$ and $\dot{y}(0)$. The closed-loop system \eqref{ClosedLoopSystem} is also reformulated as a first-order DAE and solved using \texttt{ode15s} with consistent initial data.
As shown in Fig. \ref{fig4a}, the proposed feedback shifts the characteristic roots of the open-loop system into the open left half-plane. Further, this stabilization is explicitly reflected in Figs. \ref{fig4b} and \ref{fig4c}: the output-norm trajectories exhibit a clear trend toward divergence for \eqref{OpenLoopSystem} (Fig. \ref{fig4b}), whereas those for \eqref{ClosedLoopSystem} appear to tend towards zero (Fig. \ref{fig4c}).

\begin{figure}[!htbp]
    \centering
    \begin{minipage}[htbp]{0.3\textwidth}  
        \centering  
        \includegraphics[width=\linewidth]{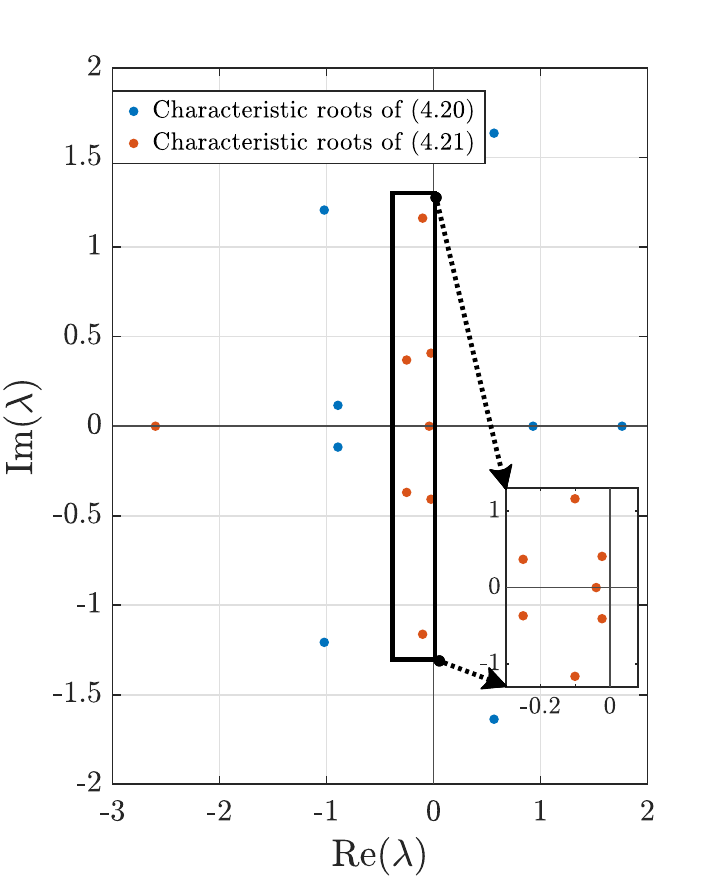}  
        \subcaption{}
        \label{fig4a}
    \end{minipage}
    \hfill
    \begin{minipage}[htbp]{0.3\textwidth}  
        \centering  
        \includegraphics[width=\linewidth]{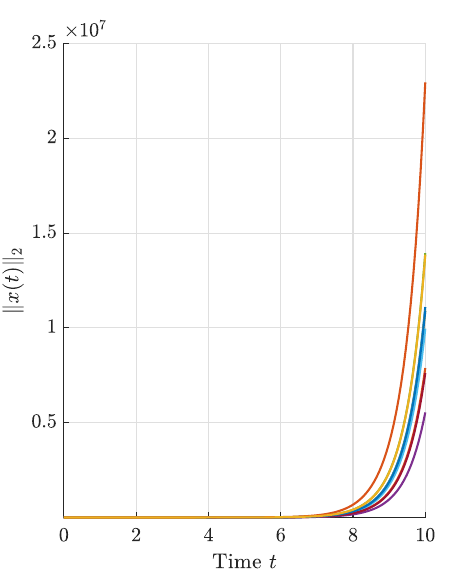}  \subcaption{}
        \label{fig4b}
    \end{minipage}
   \hfill
    \begin{minipage}[htbp]{0.3\textwidth}
        \centering  
        \includegraphics[width=\linewidth]{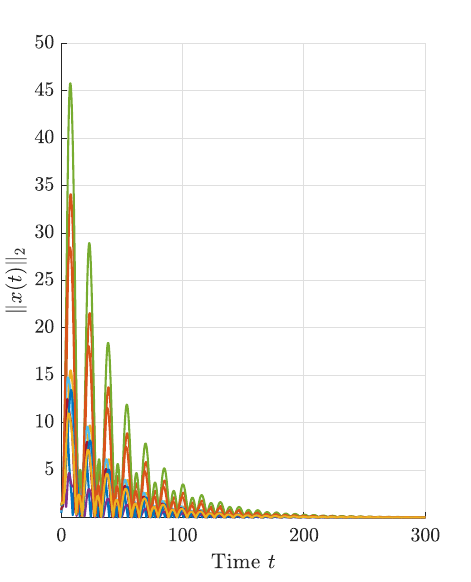}
        \subcaption{}
        \label{fig4c}
    \end{minipage}
    \caption{Fig. \ref{fig4a} compares the characteristic roots of \eqref{OpenLoopSystem} and \eqref{ClosedLoopSystem} in the complex plane, while Figs. \ref{fig4b} and \ref{fig4c} show the corresponding output-norm trajectories $\|y(t)\|_2$ for \eqref{OpenLoopSystem} and \eqref{ClosedLoopSystem}, respectively, under ten  randomly generated  initial conditions.
}
    \label{fig4}
\end{figure}
\end{example}

\begin{example}
We next evaluate the computational performance of the proposed stability
testing method based on Theorem \ref{ThmHurwitz}. In each trial, we generate a
diagonal matrix polynomial $\Lambda(\lambda)\in\mathbb{R}[\lambda]^{p\times p}$,
whose diagonal entries are Kac polynomials  \cite{Kac} with coefficients independently
sampled from the uniform distribution on $(0,1)$ and degrees independently and
uniformly drawn from $\{1,2,3\}$, and a real orthogonal matrix $U$ via the
$QR$ factorization of a $p\times p$ random matrix with i.i.d. $\mathcal N(0,1)$ entries.
The test polynomial $P(\lambda)=U\Lambda(\lambda)$ is dense and column
reduced, has Hermitian Markov parameters, and its leading coefficient is
singular whenever some $d_j<3$.

Each diagonal entry is required to be Hurwitz stable: for degrees $d=1,2$
this is automatic since all coefficients are positive, while for degree
$d=3$, writing the drawn polynomial as
$a_0\lambda^3+a_1\lambda^2+a_2\lambda+a_3$, the Routh--Hurwitz criterion
\cite[Chapter~XV, p.~194]{Gan} reduces to $a_1a_2>a_0a_3$, and the
coefficients are redrawn until it holds (each draw succeeds with probability
$1/2$, as $a_1a_2$ and $a_0a_3$ are i.i.d.). Consequently, every instance
$P(\lambda)$ is Hurwitz stable, with its stability status known a priori for
validation and its Markov parameters guaranteed to be Hermitian, so that
Theorem \ref{ThmHurwitz} applies and the Cholesky factorizations always run
to completion, i.e., each trial incurs the full computational cost. We note
that the hidden diagonal structure of $P(\lambda)$ is not exploited by our proposed 
method in any way.

The rectangular block Hankel matrices $\mathbf{H}_0^{\mathscr S}$ and
$\mathbf{H}_1^{\mathscr S}$ are constructed according to
\ref{step1}--\ref{step5}.
Their positive definiteness is then assessed via Cholesky factorization, whose
outcomes are certified a posteriori based on the rounding-error analysis in
\cite{High} (see also the certification in \cite[Section~5]{NZ}). This
completes the stability test based on Theorem \ref{ThmHurwitz}. For
comparison, MATLAB's \texttt{polyeig} function is used to compute the
eigenvalues of $P(\lambda)$, with stability determined by their real parts.
For each dimension $p$, $100$ independent trials are performed, and the
average runtime of each method is recorded.

In Figure \ref{fig5b}, the runtimes of the two methods are compared for
$p\in\{100,250,500,750,1000\}$. Over the tested range, the method based on
Theorem \ref{ThmHurwitz} (red) consistently outperforms \texttt{polyeig}
(blue), with an increasingly pronounced advantage as the matrix size grows,
indicating a clear potential of the proposed criterion for large-scale
stability testing.

Figure \ref{fig5a} shows the runtime distribution across four computational
stages for $p\in\{500,750,1000\}$, namely preprocessing (pink), adaptive ‌column-wise Li\'{e}nard-Chipart splitting and Markov parameter construction (\ref{step1}--\ref{step3},
yellow), column index set determination and block Hankel matrix construction
(\ref{step4}--\ref{step5}, orange), and positive definiteness verification
(Cholesky, blue). The results show that preprocessing and
\ref{step1}--\ref{step3} account for the largest share of the total runtime,
whereas \ref{step4}--\ref{step5} and the Cholesky factorization incur
substantially lower computational costs.

\begin{figure}[!htbp]
    \centering
    \begin{minipage}[htbp]{0.48\textwidth}
        \centering  
        \includegraphics[width=\linewidth]{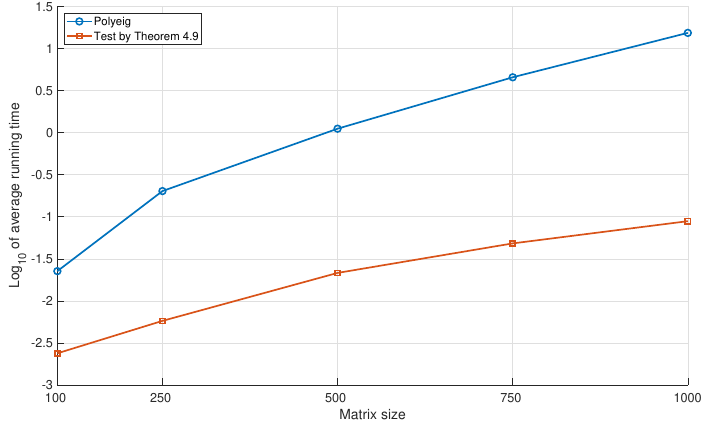}
        \subcaption{Runtime comparison between  \texttt{polyeig} and test by Theorem~\ref{ThmHurwitz}}
        \label{fig5b}
    \end{minipage}
    \hfill
    \begin{minipage}[htbp]{0.48\textwidth}  
        \centering  
        \includegraphics[width=\linewidth]{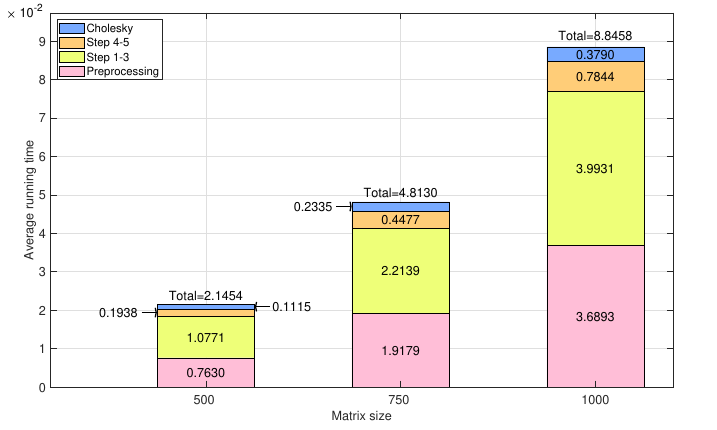}  
        \subcaption{Breakdown of runtime when varying matrix size}
        \label{fig5a}
    \end{minipage}
    \caption{}
    \label{fig5}
\end{figure}
\end{example}

\section*{Acknowledgement}
The authors thank  Professor Vanni Noferini  for in-person discussions, which provided invaluable theoretical and computational insights for this research.
The authors thank Yuan Ke for the contribution to numerical experiments and Yipeng Zhang for the fulfillment of Fig. \ref{fig:main1}. 
This work was supported by the Scientific Research Fund from the National Natural Science Foundation of China (12401489), the Beijing Natural Science Foundation (1244044) and Beijing Normal University at Zhuhai (111032119).

\appendix

\section{Infinite companion matrix}\label{SecApp}

This appendix recalls the definition of the infinite companion matrix of row reduced matrix polynomials, as introduced in \cite{BPV1999,BPV}.

For a regular matrix polynomial $D(\lambda)\in \mathbb C[\lambda]^{p\times p}$, the remainder operator $\mathcal R_D$ on 
$ \mathbb C[\lambda]^{p\times 1}$ is defined as
$$
\mathcal R_D(f(\lambda)):=D(\lambda) 
\varPi_- (D(\lambda)^{-1} f(\lambda)), 
$$
where $\varPi_-$ stands for the operator mapping from any rational matrix $r(\lambda)\in \mathbb C(\lambda)^{p\times 1}$ to its strictly proper part (i.e., the rational matrix whose elements are the strictly proper parts of the corresponding elements in $r(\lambda)$). 
The remainder space $\mathscr R_D$ is defined as the range of $\mathcal R_D$.

Let 
$
\mathscr B:=(\beta_t)_{t=0}^{\infty}
$ be the monomial basis of the linear space $\mathbb C[\lambda]^{p\times 1}$,
where 
$$
\beta_t(\lambda):=e_k \lambda^j,\quad t=jp+k,\ 0\leq k<p,\ j=0,1,\ldots, 
$$
and $e_k$ ($k=0,\ldots,p-1$) are standard basis vectors. 
The infinite matrix representation of the operator $\mathcal R_D$ with respect to the basis 
$\mathscr B$ is called the extended infinite companion matrix of $D(\lambda)$ and denoted by $\wtilde C_{\infty}(D)$. In particular, when $D(\lambda)$ is row reduced, $(\beta_t)_{t\in \mathcal I_r(D)}$ forms a basis of $\mathscr R_D$ (see, e.g., \cite[Lemma 7]{BPV}). 
The infinite companion matrix $C_{\infty}(D)$ is defined as the matrix representation of the operator $\mathcal R_D$: $\mathbb C[\lambda]^{p\times 1}\rightarrow \mathscr R_D$ with respect to the bases $\mathscr B$ and $(\beta_t)_{t\in \mathcal I_r(D)}$. Equivalently,  $C_{\infty}(D)=\wtilde C_{\infty}(D)_{\mathcal I_r(D)}$.  \cite[Theorem 8]{BPV1999} reveals a connection between $\wtilde C_{\infty}(D)$ and the infinite Hankel matrix $\mathscr{H}_{\infty}(D^{-1})$.

\begin{proposition}\label{ProCinftyD}
Let $D(\lambda)\in \mathbb C[\lambda]^{p\times p}$ be regular with the representation 
$$
D(\lambda)=\sum_{k=0}^{\infty} D_k \lambda^k\ (D_k=0\ {\rm if}\ k>\deg D).
$$
Then 
$$
\wtilde C_{\infty}(D)=\Delta(D)\mathscr{H}_{\infty}(D^{-1}),
$$
where $\Delta(D)=\big(D_{i+j+1}\big)_{i,j=0}^{\infty}$. 
\end{proposition}

\bibliographystyle{amsplain}

\begin{thebibliography}{99}

\bibitem{AnJu} B.D.O. Anderson, E. Jury, \textit{Generalized Bezoutian and Sylvester matrices in multivariable linear control}, IEEE Trans. Autom. Control AC-21 (1976) 551--556.

\bibitem{BPV1999}
M. Van Barel, V. Pt\u{a}k, Z. Vavr\u{r}\'{i}n, \textit{Extending the notions of companion and infinite companion to matrix polynomials},
Linear Algebra Appl. 290 (1999) 61--94.

\bibitem{BPV}
M.Van Barel, V. Pt\u{a}k, Z. Vavr\u{r}\'{i}n, 
\textit{B\'{e}zout and Hankel matrices associated with row reduced matrix polynomials, Barnett-type formulas}, Linear Algebra Appl. 332--334 (2001)  583--606.


\bibitem{BH} 
D. Bindel, A. Hood,  \textit{Localization theorems for nonlinear eigenvalue
problems}, SIAM J. Matrix Anal. Appl. 34 (2013) 1728--1749.

\bibitem{BA} R.R. Bitmead, B.D.O. Anderson, \textit{The matrix Cauchy index: properties and
applications}, SIAM J. Appl. Math. 33 (1977) 655--672.


\bibitem{CIK}  D.O. Chernyshov, A.V. Ivlev, A.M. Kiselev,   \textit{Self-consistent theory of cosmic ray penetration into molecular clouds: Relativistic case}, Phys. Rev. D 110 (4) (2024) 043012.

\bibitem{CM} D. Chu, V. Mehrmann, \textit{Asymptotic stability and strict passivity of port-Hamiltonian descriptor systems via state feedback}, Syst. Control Lett. 202 (2025) 106116.

\bibitem{DNQD} F. Dopico, V. Noferini, M.C. Quintana, P. Van Dooren, \textit{Para-Hermitian rational matrices}, SIAM J. Matrix Anal. Appl. 45 (4) (2024) 2339--2359.

\bibitem{DyV}H. Dym, D. Volok, \textit{Zero distribution of matrix polynomials}, Linear Algebra Appl. 425 (2007) 714--738.


\bibitem{HAP} D. Henrion, D. Arzelier,  D. Peaucelle, \textit{Positive polynomial matrices and
improved LMI robustness conditions}, Automatica 39 (2003) 1479--1485.

\bibitem{HAPS} D. Henrion, D. Arzelier, D. Peaucelle, M. Sebek, \textit{An LMI condition for robust
stability of polynomial matrix polytopes}, Automatica 37 (2001) 461--468.

\bibitem{Hermite} C. Hermite, \textit{Sur l’introduction des variables continues dans la th\'{e}orie des nombres}, J. Reine Angew. Math. 41 (1851)
191--216.

\bibitem{Hu} G. Hu, \textit{Feedback controller design in the frequency domain for high linear}
$n$-\textit{order systems}, J. Comput. Appl. Math. 475 (2026) 117052.


\bibitem{Gan} F.R. Gantmacher, \textit{The Theory of Matrices}, Vols. 1, 2, translated by K.A. Hirsch, Chelsea Publishing Co., New York, 1959.

\bibitem{Gal} R. Galindo, \textit{Stabilisation of matrix polynomials}, Int. J.
    Control
88 (2015) 1925--1932.

\bibitem{GJV} P. Giorgi, C.P. Jeannerod, G. Villard, \textit{On the complexity of polynomial matrix computations}, in: Proc. Int. Symp. Symbolic Algebr. Comput. (ISSAC), 2003, pp. 135--142.


\bibitem{GLRMP} I. Gohberg, P. Lancaster, L. Rodman, \textit{Matrix Polynomials}, Academic Press, New York, 1982.

\bibitem{GKLR}I. Gohberg, M.A. Kaashoek, L. Lerer, L. Rodman, \textit{Common multiples and common divisors of matrix polynomials, II. Vandermonde and resultant matrices}, Linear Multilinear Algebra 12 (1982) 159--203.

\bibitem{GHN} P. Grindrod, D. Higham, V. Noferini, \textit{The deformed graph Laplacian and its applications to network centrality analysis}, SIAM J. Matrix Anal. Appl. 39 (1) (2018) 310--341.

\bibitem{Kai} T. Kailath, \textit{Linear Systems}, Prentice-Hall, Englewood Cliffs, N.J., 1980.

\bibitem{KG}
V. Kosti\'c, D. Garda\u{s}evi\'c, \textit{On the Ger\u{s}gorin-type localizations for nonlinear eigenvalue problems}, Appl. Math. Comput. 37 (2018) 179--189.

\bibitem{KM} P. Kunkel, V. Mehrmann, \textit{Differential-Algebraic Equations: Analysis and Numerical Solution}, European Mathematical Society, 2006.

\bibitem{LNZ} G. Labahn, V. Neiger, W. Zhou, \textit{Fast, deterministic computation of the
Hermite normal form and determinant of a
polynomial matrix}, J. Complex. 42 (2017) 44--71.

\bibitem{LPJ} D.H. Lee, J.B. Park, Y.H. Joo,
    \textit{A less conservative LMI condition for robust $\mathcal D$-stability of polynomial matrix
        polytopes--- a projection approach}, IEEE Trans. Automat. Control 56 (2011) 868--873.



\bibitem{LRT} L. Lerer, L. Rodman,  M. Tismenetsky, \textit{Inertia theorems for matrix polynomials}, Linear Multilinear Algebra 30 (1991) 157--182.




\bibitem{LT82} L. Lerer, M. Tismenetsky, \textit{The Bezoutian and the eigenvalue-separation problem for matrix polynomials}, Integral Equations Oper. Theory 5 (1982) 386--445.


\bibitem{Mat} D. Matignon, \textit{Stability results for fractional differential equations with applications to control processing}, in: Computational Engineering in
Systems Applications, Vol. 2, Lille, France, 1996, pp. 963--968.

\bibitem{MMW} C. Mehl, V. Mehrmann, M. Wojtylak, \textit{Matrix pencils with coefficients that have
positive semidefinite Hermitian parts}, SIAM J. Matrix Anal. Appl. 43 (3) (2022) 1186--1212.

\bibitem{NRS} V. Neiger, J. Rosenkilde, G. Solomatov, \textit{Computing Popov and Hermite forms of rectangular polynomial matrices}, in: Proc. Int. Symp. Symbolic Algebr. Comput. (ISSAC), 2018,
pp. 295--302.


\bibitem{NN} V. Noferini, Y. Nakatsukasa, \textit{Inertia laws and localization of real eigenvalues for generalized indefinite eigenvalue problems},  Linear Algebra  Appl. 578 (2019) 272--296.

\bibitem{NV} V. Noferini, P. Van Dooren, \textit{Revisiting the matrix polynomial greatest common divisor}, SIAM J. Matrix Anal. Appl. 44 (3) (2023) 1164--1188.

\bibitem{Popov} V.M. Popov,  \textit{Invariant description of linear, time-invariant controllable
systems}, SIAM J. Control 10 (2) (1972) 252--264.



\bibitem{RB} N. Roy, S. Bora, \textit{Locating eigenvalues of quadratic matrix polynomials}, Linear
Algebra Appl. 649 (2022) 452--490.

\bibitem{Rudin} W. Rudin, \textit{Real and Complex Analysis}, McGraw-Hill, Inc., 1987.

\bibitem{SD} N. 
Schl\"{u}ter, M.S. Darup, \textit{On the stability of linear dynamic controllers with integer coefficients}, IEEE Trans. Automat. Control 67 (10) (2022) 5610–5613.

\bibitem{DY2006} G-R. Duan, H-H. Yu, \textit{Complete eigenstructure assignment
in high-order descriptor linear systems
via proportional plus derivative state feedback}, Proceedings of the 6th World Congress on Intelligent Control and Automation (2006).

\bibitem{Kac} M. Kac, \textit{On the average number of real roots of a random algebraic equation (II)}, Proc. London Math. Soc.  50 (2) (1948) 390--408.


\bibitem{LAF} E. López-Flores, B. Aguirre-Hernández, M.E. Frías-Armenta, \textit{Stability analysis of polynomials with an approach of differential topology}, Asian J. Control, 27 (2) (2025) 1068--1074.

\bibitem{ALCRMD} B. Aguirre-Hernández, C.A. Loredo-Villalobos, D. Castillo-Fernández, C. Romero-Meléndez,  R.F.  Monroy-Pérez, E.C. Díaz-González, \textit{Polynomial approaches in control theory},  
   Bol. Soc. Mat. Mex. 32 (3) (2026) 77.
   
\bibitem{Aguirre} B. Aguirre-Hernández,  \textit{Biquadratic transformations for checking the Hurwitz stability via 
   the Schur stability}, Asian J. Control, 0 (2026) 1--4.

\bibitem{ADLS} B. Aguirre-Hernández, E.C. Díaz-González, C.A. Loredo-Villalobos, T. Sánchez-Alcántara, \textit{Continuous and discrete stability criteria comparison}, J. Dyn. Syst. Complex. 1 (1) (2025) 12-19.

\bibitem{FD} H. Fatoorehchi, S. Djilali, \textit{Stability analysis of linear time-invariant dynamic systems using the matrix sign function and the Adomian decomposition method}, Int. J. Dyn. Control 11 (2) (2023) 593--604.
   
\bibitem{FF} H. Fatoorehchi, P.L. Fosbøl, \textit{On the maximum horizontal distance of a matrix's eigenvalues from the imaginary axis: a new criterion for control system stability}, J. Frankl. Inst. 363 (1) (2025) 108258.

\bibitem{FE} H. Fatoorehchi, M. Ehrhardt, \textit{A novel technique for stability analysis of linear  
   time-invariant and nonlinear continuous-time control systems based on Markov parameters and Hankel matrices}, Trans. Inst. Meas. Control 46 (10) (2024) 1910--1923.


\bibitem{BAK} S.P. Bhattacharyya, A. Datta, L.H. Keel, \textit{Linear Control Theory:
Structure, Robustness, and Optimization}, CRC Press, Boca Raton, 2018.

\bibitem{HP} D. Hinrichsen, A.J. Pritchard, \textit{Mathematical Systems Theory I: Modelling,
State Space Analysis, Stability and Robustness}, Springer, Berlin, 2005.

\bibitem{Leo} G.A. Leonov, Mathematical \textit{Problems of Control Theory: An Introduction},
World Scientific, Singapore, 2001.

\bibitem{Bar} S. Barnett, \textit{Polynomials and Linear Control Systems}, Marcel Dekker, New
York, 1983.

\bibitem{Kush} O.Y. Kushel, \textit{Unifying matrix stability concepts with a view to applications}, SIAM Review 61 (4) (2019) 643–729.

\bibitem{Lyapu} A.M. Lyapunov, \textit{The general problem of the stability of motion}, Int. J. Control 55 (3) (1992) 531–773.

\bibitem{HT}O. Holtz, M. Tyaglov, \textit{Structured matrices, continued fractions, and root localization of polynomials}, SIAM Review 54 (3) (2012) 421–509.

\bibitem{GV} G.H. Golub, C.F. Van Loan, \textit{Matrix Computations}, 4th ed., Johns Hopkins University Press, Baltimore, 2013.

\bibitem{High} N.J. Higham, \textit{Accuracy and Stability of Numerical Algorithms}, 2nd ed., SIAM, Philadelphia, 2002.

\bibitem{NZ} V. Noferini, X. Zhan, \textit{Fast stability tests for Hermitian matrix polynomials}, preprint, https://arxiv.org/abs/2607.15636.

\bibitem{ZBH} X. Zhan, B. Ban, Y. Hu, \textit{On the quasi-stability criteria of monic matrix polynomials}, J. Comput. Appl. Math. 438 (2024) 115560.

\bibitem{XD}  X. Zhan,  A. Dyachenko, \textit{On  generalization of classical Hurwitz stability criteria  for matrix polynomials}, J. Comput. Appl. Math. 383 (2021) 113113.

\bibitem{ZH} X. Zhan, Y. Hu, \textit{On the relation between Hurwitz stability of matrix polynomials and matrix-valued Stieltjes functions}, J. Comput. Appl. Math. 417 (2023) 114614.


\end{thebibliography}

\end{document}